\documentclass[11pt, twoside, leqno, a4paper]{amsart}  
\usepackage{lipsum}
\usepackage{amsfonts}
\usepackage{graphicx}
\usepackage{epstopdf}
\usepackage{algorithmic}
\usepackage{calligra}
\usepackage{amsfonts,amsmath,amsthm,amssymb}
\usepackage{mathtools}
\usepackage{hyperref}
\usepackage{autonum}
\usepackage{hhline}
\usepackage{array}
\usepackage{diagbox}
\usepackage{tcolorbox}
\usepackage{mdframed}
\usepackage{multicol}
\usepackage{graphicx}
\usepackage{subcaption}
\usepackage{moreverb}
\usepackage{bbm}
\usepackage[margin=1.4in]{geometry}
\usepackage{todonotes}
\usepackage{scalerel,amssymb}
\allowdisplaybreaks
\usepackage{mathrsfs}  
\usepackage{lineno}
\usepackage{todonotes}
\usepackage{tikz}
\usepackage{pgfplots}
\pgfplotsset{compat=1.7}
\usetikzlibrary{arrows.meta}
\usepackage{wasysym}
\usepackage[numbers,sort&compress]{natbib}
\usetikzlibrary{shapes,backgrounds}
\newcommand{\ddt}{\frac{\textup{d}}{\textup{d}t}}

\newcommand{\dt}{\, \textup{d} t}
\newcommand{\ds}{\, \textup{d} s }
\newcommand{\dx}{\, \textup{d} x}

\newcommand{\sca}[2]{\langle #1, \; #2 \rangle}

\newcommand{\dxs}{\, \textup{d}x \textup{d}s}
\newcommand{\dxt}{\, \textup{d}x \textup{d}t}

\newcommand{\inttO}{\int_0^t \int_{\Omega}}
\newcommand{\intT}{\int_0^t}
\newcommand{\intO}{\int_{\Omega}}

\newcommand{\ulal}{\underline{\alpha}}
\newcommand{\olal}{\overline{\alpha}}

\newcommand{\R}{\mathbb{R}} 
\newcommand{\N}{\mathbb{N}} 

\newcounter{rownumber}

\DeclareMathOperator*{\esssup}{ess\,sup}

\newcommand{\nLtwo}[1]{\| #1 \|_{L^2}}
\newcommand{\nHone}[1]{\| #1 \|_{H^1}}
\newcommand{\nLinf}[1]{\| #1 \|_{L^\infty}}
\newcommand{\nLtwoLtwo}[1]{\| #1 \|_{L^2(L^2)}}
\newcommand{\nLtwotLtwo}[1]{\| #1 \|_{L^2(L^2)}}

\newcommand{\nLtwotLfour}[1]{\| #1 \|_{L^2(L^4)}}
\newcommand{\nLinfLtwo}[1]{\| #1 \|_{L^\infty(L^2)}}

\newtheorem{theorem}{Theorem}

\newtheorem{proposition}{Proposition}
\newtheorem*{assumption*}{Assumptions on the memory kernel}

\newtheorem{remark}{Remark}
\numberwithin{lemma}{section}
\numberwithin{proposition}{section}
\numberwithin{theorem}{section}
\numberwithin{equation}{section}
\makeatletter
\newcommand{\leqnomode}{\tagsleft@true}
\newcommand{\reqnomode}{\tagsleft@false}
\makeatother   
\usepackage{aliascnt}
\numberwithin{equation}{section}

\numberwithin{thm}{section}

\newaliascnt{corollary}{thm}
\newtheorem{corollary}[corollary]{Corollary}
\aliascntresetthe{corollary}

\title[Shape and topology optimization of acoustic waves]{A phase-field approach to shape and topology\\[1mm] optimization of acoustic waves\\[1mm] in dissipative media}



\author[H. Garcke]{Harald Garcke$^\dagger$}
\thanks{$^\dagger$Fakult\"at f\"ur Mathematik,
	Universit\"at Regensburg, 93040 Regensburg, Germany (\href{harald.garcke@mathematik.uni-regensburg.de}{harald.garcke@mathematik.uni-regensburg.de})}

\author[S. Mitra]{Sourav Mitra$^\ddag$}
\thanks{$^\ddag$Institute of Mathematics, University of W\"urzburg, 97074 W\"urzburg, Germany, (\href{sourav.mitra@mathematik.uni-wuerzburg.de}{sourav.mitra@mathematik.uni-wuerzburg.de})}

\author[V. Nikoli\'c]{Vanja Nikoli\'c$^\S$}
\thanks{$^\S$Department of Mathematics, 	Radboud University, 6525 AJ Nijmegen, The Netherlands
(\href{vanja.nikolic@ru.nl}{vanja.nikolic@ru.nl})}

\begin{document}
\maketitle

\begin{abstract}
We investigate the problem of finding the optimal shape and topology of a system of acoustic lenses in a dissipative medium. The sound propagation is governed by a general semilinear strongly damped wave equation. We introduce a phase-field formulation of this problem through diffuse interfaces between the lenses and the surrounding fluid. The resulting formulation is shown to be well-posed and we prove that the corresponding optimization problem   has a minimizer. By analyzing properties of the reduced objective functional and well-posedness of the adjoint problem, we rigorously derive first-order optimality conditions for this problem. Additionally, we consider the $\Gamma$-limit of the reduced objective functional and in this way establish a relation between the diffuse interface problem and a perimeter-regularized sharp interface shape optimization problem.
\end{abstract} 
\section{Introduction}
Optimization of acoustic wave propagation is of immediate interest in numerous medical and industrial applications. In medical uses of ultrasound, acoustic lenses focus sound waves in a way analogous to optical lenses at the part of the body being treated or examined~\cite{spadoni2010generation, yoshizawa2009high}. Their design directly influences the achieved acoustic pressure levels in the focal region and, in turn, the quality and safety of these procedures. In underwater imaging, a system of acoustic cameras creates images that can, unlike with optical lenses, still be obtained in low-visibility waters~\cite{tran2017shape, belcher2001object, belcher1999beamforming, mueller2006video}. They are used in, e.g., tracking the work of divers or remotely operated vehicles~\cite{belcher1999beamforming}.  Shape and topology optimization of their design is expected to lead to a better quality of images. \\
\indent In this work, we study the problem of optimizing the shape of a system of acoustic lenses in dissipative media, which will also allow for topological changes. Under sufficiently high frequencies (as in the above-mentioned ultrasonic applications) and/or intensities, the sound propagation is governed by nonlinear wave equations. We thus use a general semilinear strongly damped wave equation as a model of nonlinear acoustic propagation to match the desired acoustic pressure distribution in the focal region. The model in question can be understood as a semi-linearization of the classical Westervelt equation of nonlinear acoustics~\cite{westervelt1963parametric}. This semi-linearization allows us to obtain existence of solutions with a relatively low regularity in space; that is, with acoustic pressure in $L^\infty(0,T; H^1(\Omega))$.
The low space regularity is due to inhomogeneities in space and hence a higher space regularity cannot be expected. The low space regularity will give rise to several analytical difficulties. \\
\indent We propose and analyze a phase-field method with regularization by a Ginzburg-Landau energy, where the phase-field function models the transition between the fluid and lenses. The problem is then formulated as an optimal control problem with control in the medium parameters, that is, the speed of sound, sound diffusivity, and the nonlinearity coefficient. We prove that this diffuse interface formulation is well-posed and that the optimization problem  admits a minimizer. We also analyze the differentiability of the reduced objective functional and well-posedness of the adjoint problem, thanks to which we can derive first-order optimality conditions for this problem. Additionally, we are able to relate the diffuse interface problem to the perimeter-regularized sharp interface shape optimization problem  through the $\Gamma$-limit of the reduced objective functional.\\
\indent To the best of our knowledge, a phase-field approach has not been considered before in a nonlinear acoustic setting. For linear time-harmonic propagation, a numerical algorithm based on a phase-field method has been investigated in~\cite{tran2017shape} relying on formal calculations. The sharp interface problem has been analyzed in~\cite{nikolic2017sensitivity} subject to the Westervelt equation with strong nonlinear damping, with restrictions in terms of the space of admissible shapes. A numerical algorithm with sound propagation modeled by the Westervelt equation and in part formal computations has been developed in~\cite{muhr2017isogeometric}.    \\
\indent We structure our exposition as follows. In Section~\ref{Sec:ProblemSetting}, we discuss the modeling aspects of sound propagation through fluid with acoustic lenses and introduce a phase-field formulation of the related shape optimization problem. Section~\ref{Sec:AnalysisState} is dedicated to the well-posedness analysis of the state problem. In Section~\ref{Sec:OptimalControlFormulation}, we analyze the control-to-state operator and prove its Fr\'echet differentiability. Section~ \ref{Sec:ExistenceMinimizer} deals with the existence of a minimizer. In Section~\ref{Sec:Adjoint} we analyze the adjoint problem, which, together with the previous results, enables us to derive the first-order optimality conditions in Section~\ref{Sec:FirstOrderOpt}. Finally, in Section~\ref{Sec:SharpInterface} we study the $\Gamma$-convergence of the reduced objective functional, which allows us to relate the phase-field problem to the sharp interface shape optimization problem with a perimeter regularization. 
\section{Problem setting and a phase-field approach} \label{Sec:ProblemSetting}
We consider an acoustic lens system in a thermoviscous fluid. A number of acoustic lenses $\Omega_{l,1}, \ldots, \Omega_{l, n}$ of the same material are immersed in an acoustic fluid $\Omega_f$, $n \in \N$; see Figure~\ref{Fig:SharpInterface}. The material parameters corresponding to the lens are given by $(c_l, b_l, k_l)$ and to the fluid by $(c_f, b_f, k_f)$. Here $c_i>0$ denotes the speed of sound, $b_i>0$ the so-called sound diffusivity, and $k_i \in \R$ is the nonlinearity coefficient, where $i \in \{l, f\}$. Typical values of the speed of sound and sound diffusivity in different media can be found in, e.g.,~\cite{dunn2015springer, hamilton1998nonlinear}. \\
\indent The goal is to determine the number and shape of acoustic lenses such that we reach the desired pressure distribution in some region of interest $D \subset \Omega$, where $\Omega \subset \R^d$, $d \in \{2,3\}$, is a hold-all domain, assumed to be Lipschitz regular. Let $T>0$ denote the final time of propagation.  Assuming that we have a high-intensity or high-frequency sound source, the propagation of sound waves is nonlinear. We can obtain the pressure field $u$ by solving 
\begin{equation}  \label{Westervelt_eq_simplified}
\begin{aligned}
\alpha(x, t)u_{tt}-\textup{div}(c^2\nabla u)- \textup{div}(b \nabla u_t)= f(u_t) \quad \text{on } \ \Omega \times (0,T)
\end{aligned}
\end{equation}
with the right-hand side nonlinearity given by
\[
f(u_t)=2k u_t^2.
\]  
The medium parameters are piecewise constant functions, defined as
\begin{equation} \label{assumption_coefficients}
\begin{aligned}
    c=&\, c_l \chi_{\Omega_l}+c_f(1-\chi_{\Omega_l}),\\
    b=& \, b_l \chi_{\Omega_l}+b_f(1-\chi_{\Omega_l}),\\ 
    k=& \, k_l \chi_{\Omega_l}+k_f(1-\chi_{\Omega_l}), 
\end{aligned}    
\end{equation}
with $\Omega_l=\displaystyle \bigcup_{j=1}^n  \, \Omega_{l, j}$. We assume that the coefficient $\alpha$ does not degenerate, that is, we assume that there exist $\ulal$, $\olal >0$, such that
\begin{equation} \label{non-degeneracy}
\ulal \leqslant \alpha(x,t) \leqslant \olal \qquad \text{a.e.\ in } \Omega \times (0,T).
\end{equation}
The sound waves are excited via boundary in form of Neumann boundary conditions
\begin{equation}
\begin{aligned}
c^2\frac{\partial u}{\partial n}+b \frac{\partial u_t}{\partial n}=g \ \ \text{on} \ \ \Gamma =\partial \Omega
\end{aligned}
\end{equation}
where $n$ denotes the unit outward normal to $\Gamma$ and the problem is additionally supplemented with the initial conditions
\begin{equation}
(u,u_t)\vert_{t=0}=(u_0,u_1).
\end{equation}
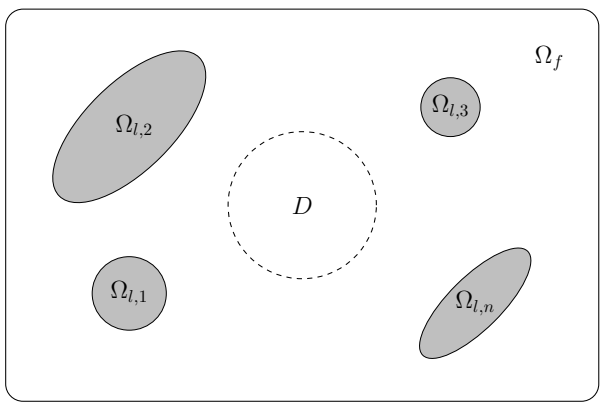
\begin{figure}[h!]
	\centering
	\scalebox{0.65}{\begin{tikzpicture}
\draw[black, fill = white, fill opacity = 0.5, semithick, even odd rule,rounded corners=10]
            (0,0) rectangle (12,8);
\draw[black, fill = white, dashed,  fill opacity = 0.5, semithick] 
(6, 4) circle (1.5);
\node[circle, draw, fill=gray, fill opacity = 0.5,  minimum size = 1.2cm] at (9, 6) {};
\node[circle, draw, fill=gray, fill opacity = 0.5,  minimum size = 1.5cm] at (2.5, 2.2) {};
\node at (11, 7) {\Large $\Omega_f$};
\node at (6, 4) {\Large $D$};
\node[ellipse, draw,fill=gray, fill opacity = 0.5, text height = 1cm, minimum width = 4cm, rotate=45] at (2.5, 5.6) {};
\node[ellipse, draw,fill=gray, fill opacity = 0.5, text height = 0.5cm, minimum width = 3cm, rotate=45] at (9.5, 2) {};
\node at (2.6, 5.6) {\Large $\Omega_{l,2}$};
\node at (2.5, 2.2) {\Large $\Omega_{l,1}$};
\node at (9, 6) {\Large $\Omega_{l, 3 }$};
\node at (9.5, 2) {\Large $\Omega_{l, n}$};
\end{tikzpicture}}
	\caption{The acoustic lens system.} \label{Fig:SharpInterface}
\end{figure}  
~\\
Equation \eqref{Westervelt_eq_simplified} can be understood as a semi-linearization of the Westervelt equation of nonlinear acoustics obtained by freezing the term $\alpha(u)=1-2ku$. We refer to~\cite{westervelt1963parametric} for its derivation in lossless form and, for example, \cite{kaltenbacher2009global, kaltenbacher2011well} for its analysis in homogeneous media. This semi-linearization admits solutions $u \in L^\infty(0,T; H^1(\Omega))$ for sufficiently small data, as we will show. However, due to the heterogeneous medium setting in the present work, a higher space regularity cannot be expected, which will lead to several analytical difficulties in the analysis. 
\subsection{A phase-field approach}
Similarly to~\cite{tran2017shape}, where a numerical phase-field approach for shape optimization of acoustics lenses has been investigated under the assumption of time-harmonic wave propagation, we introduce a continuous material representation between lenses and fluid by employing diffuse interfaces $\xi_i$, $i \in [1,n]$, with thickness proportional to $\varepsilon>0$. We next define a partition 
$$\Omega=\Omega_{f}\cup\overline{\xi}\cup\Omega_{l}$$
of $\Omega,$ where $\xi=\bigcup_{i=1}^n \xi_i$
and then introduce a phase-field function $\varphi$, such that
\begin{equation}
    \begin{aligned}
    &\varphi(x)=1 \ \ \text{for } x \in \Omega_f, \\
    0\leqslant &\, \varphi(x)\leqslant 1  \ \ \text{for } x \in \overline{\xi},\\
    &\varphi(x)= 0 \ \ \text{for } x \in \Omega_l;
    \end{aligned}
\end{equation}
see Figure~\ref{Fig:PhaseFieldInterface}.
In the phase-field setting the fluid region $\Omega_f$  and  the lens region $\Omega_l$  are hence  separated by a diffuse interface. On the diffuse interface, the material properties are interpolated with respect to the phase-field function as follows:
\begin{equation} \label{interpolated_coefficients}
    \begin{aligned}
        c^2=&\, c^2_l +\varphi(x)(c^2_f-c^2_l),\\
        b=&\, b_l +\varphi(x)(b_f-b_l),\\
        k=&\, k_l +\varphi(x)(k_f-k_l),\\        
    \end{aligned}
\end{equation}
where we assume $c_l < c_f$, $b_l < b_f$, and $k_l < k_f$. 
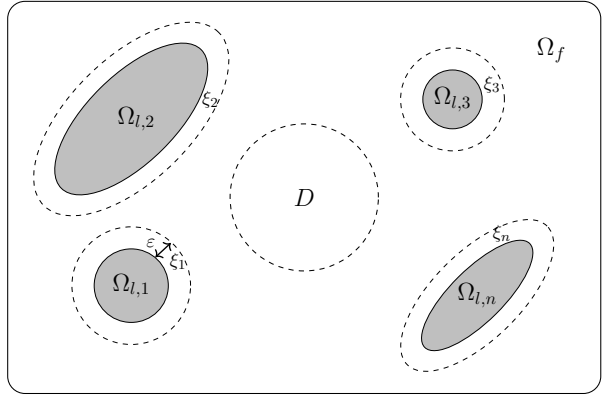
\begin{figure}[h!]
	\centering
	\scalebox{0.65}{\begin{tikzpicture}
\draw[black, fill = white, fill opacity = 0.5, semithick, even odd rule,rounded corners=10]
(0,0) rectangle (12,8);
\draw[black, fill = white, dashed,  fill opacity = 0.5, semithick, even odd rule] 
(6, 4) circle (1.5);
\node at (11, 7) {\Large $\Omega_f$};
\node at (6, 4) {\Large $D$};
\node[ellipse, draw,fill=white,  dashed, text height = 1.5cm, minimum width = 5cm, rotate=45] at (2.5, 5.6) {};
\node[ellipse, draw,fill=gray, fill opacity = 0.5, text height = 1cm, minimum width = 4cm, rotate=45] at (2.5, 5.6) {};

\node[ellipse, draw,fill=white, dashed, text height = 1cm, minimum width = 4cm, rotate=45] at (9.5, 2) {};
\node[ellipse, draw,fill=gray, fill opacity = 0.5, text height = 0.5cm, minimum width = 3cm, rotate=45] at (9.5, 2) {};

\node[circle, draw, fill=white, dashed,  minimum size = 2.1cm] at (9, 6) {};
\node[circle, draw, fill=gray, fill opacity = 0.5,  minimum size = 1.2cm] at (9, 6) {};

\node[circle, draw, fill=white, dashed,  minimum size = 2.4cm] at (2.5, 2.2) {};
\node[circle, draw, fill=gray, fill opacity = 0.5,  minimum size = 1.5cm] at (2.5, 2.2) {};
\draw[<->, thick] (2.99,2.78)--(3.3, 3.07);
\node at (2.9, 3.065) {$\varepsilon$};
\node at (2.6, 5.6) {\Large $\Omega_{l,2}$};
\node at (2.5, 2.2) {\Large $\Omega_{l,1}$};
\node at (9, 6) {\Large $\Omega_{l, 3}$};
\node at (9.5, 2) {\Large $\Omega_{l, n}$};
\node at (3.45, 2.7) {$\xi_{1}$};
\node at (10, 3.27) {$\xi_{n}$};
\node at (9.8, 6.3) {$\xi_{3}$};
\node at (4.1,6) {$\xi_{2}$};
\end{tikzpicture}}
	\caption{The acoustic lens system with a phase-field interface.}  \label{Fig:PhaseFieldInterface}
\end{figure} 
~\\
\noindent To formulate the optimization problem, we employ a tracking-type objective functional for a given desired pressure $u_{\textup{d}} \in L^2(0,T; L^2(\Omega))$. Following, for example,~\cite{ambrosio1993optimal, petersson1999some}, we use a perimeter penalization to overcome ill-posedness of the sharp interface problem. We then approximate it in the diffuse interface setting by a multiple of the Ginzburg--Landau energy $E_\varepsilon$:
\begin{equation} \label{GL_energy}
\begin{aligned}
E_\varepsilon(\varphi)=\begin{cases}  \displaystyle \int_{\Omega} \dfrac{\varepsilon}{2}|\nabla \varphi|^2+\dfrac{1}{\varepsilon}\Psi(\varphi) \textup{d}x,\quad  \quad  &\text{if } \varphi \in H^1(\Omega), \\[2mm]
+ \infty, \ &\text{otherwise}; \end{cases} 
\end{aligned}
\end{equation}
cf.~\cite{garcke2016shape, blank2014relating, blank2016sharp}. Here $\Psi$ is a double obstacle potential given by
\begin{equation}
\begin{aligned}
\Psi(\varphi) = \begin{cases} \Psi_0(\varphi) \qquad &\text{if } \ 0 \leqslant \varphi \leqslant 1, \\ + \infty, \ &\text{otherwise}, \end{cases}
\end{aligned}	
\end{equation}
with 
\[
\Psi_0(\varphi)= \frac{1}{2}\varphi(1-\varphi).
\]
The shape optimization problem then has the following phase-field formulation: 
\begin{equation} \label{objective}
\begin{aligned}
\min_{(u, \varphi)} J^\varepsilon(u, \varphi)= \dfrac12 \int_0^T \int_D (u-u_\textup{d})^2 \, \textup{d}x \textup{d}s+\gamma E_\varepsilon(\varphi)
\end{aligned}
\end{equation}
where $\gamma>0$ is a weighting parameter, with 
\begin{equation} \label{admissible_sets}
\varphi \in \Phi_{\textup{ad}}=\{\varphi \in H^1(\Omega) \cap L^\infty(\Omega): \, 0 \leqslant \varphi \leqslant 1 \ \text{a.e. in} \ \Omega\},\quad u \in U,
\end{equation}
where the admissible space $U$ is defined in \eqref{defU}, such that
\begin{equation} \label{state_constraint_strong}
\begin{aligned}
\begin{cases}
\alpha(x,t) u_{tt}-\textup{div}(c^2(\varphi) \nabla u)- \textup{div}(b(\varphi) \nabla u_t)= 2k(\varphi)u_t^2, \\[2mm]
c^2(\varphi)\dfrac{\partial u}{\partial n}+b(\varphi) \dfrac{\partial u_t}{\partial n}=g\ \ \text{on} \ \ \Gamma,\\[2mm]
(u, u_t)\vert_{t=0}=(0, 0),
\end{cases}
\end{aligned}
\end{equation}
is satisfied (in a weak sense) and the medium parameters satisfy \eqref{interpolated_coefficients}.\\
\indent The function $\varphi \in \Phi_{\textup{ad}}$ is thus the design variable with $\{x \in \Omega: \varphi(x)=1\}$ modeling the fluid region and $\{x \in \Omega: \varphi(x)=0\}$ the lenses. In what follows, we proceed to analyze this problem.
\begin{remark}
Additional constraints could be incorporated in $\Phi_{\textup{ad}}$ without affecting the upcoming analysis. For instance, it might be sensible to impose that the focal region $D$ could only be filled with fluid. In that case, we can modify the definition of  $\Phi_{\textup{ad}}$ to involve the condition \[\varphi=1 \quad \text{in} \  D.\] We might also wish to impose a volume constraint on the fluid region (that is, on $\displaystyle \int_{\Omega} \varphi \dx$) to avoid the limiting cases of the whole domain $\Omega$ being filled with fluid or occupied by a lens.
\end{remark}
\begin{remark}
In applications it is often desirable to limit the number of lenses. The parameter	 $\gamma$ which weights the interfacial energy  can be used to penalize the number of holes in the fluid. Of course a large $\gamma$ makes it energetically more costly to have many interfaces and this leads to the fact that the number of lenses will decrease for large $\gamma$. This has been used in shape and topology optimization problems for elastic structures, see e.g.\ Rupprecht \cite{rupprecht}, Figures 37 and 38.

It is more subtle to restrict the number of lenses to a fixed quantity, lets say one or two. In this case the objective functional needs to include topological information. In this context possibly  the work of Du, Liu and Wang  \cite{du2005retrieving}
 can be used.
\end{remark}
\begin{remark}
	In view of relevant ultrasound applications, it could also be interesting to involve optimal control of the boundary excitation $g$ into the above optimization problem. This would, however, require a non-trivial extension of the current analytical framework with possibly incorporating ideas from~\cite{clason2009boundary}, where boundary control problems for nonlinear acoustic models in homogeneous media have been considered.
\end{remark}
\section{Analysis of the state problem} \label{Sec:AnalysisState}
In this section, we consider the state problem \eqref{state_constraint_strong}. To formulate the results, we introduce the solution space
\begin{equation}\label{defU}
\begin{aligned}
U = \{u \in L^\infty(0,T; H^1(\Omega)): \ u_t \in  L^\infty(0,T; H^1(\Omega)), \  u_{tt} \in L^2(0,T; L^2(\Omega))\}
\end{aligned}
\end{equation}
with the corresponding norm
\begin{equation} \label{norm_U}
\begin{aligned}
\|u\|_{U}=&\, \begin{multlined}[t] \left\{\sup_{t \in (0,T)}\nHone{u(t)}^2 + \esssup_{t \in (0,T)}\nHone{u_t(t)}^2+\int_0^T\nLtwo{u_{tt}(s)}^2 \ds\right\}^{1/2}.\end{multlined}
\end{aligned}
\end{equation}
\indent In what follows, we will frequently abbreviate the norms on $L^p(\Omega)$ for $p \in [1, \infty]$, $H^1(\Omega)$, and in spaces on the boundary by omitting the spatial domain. From the context, it will be clear whether the spatial domain is $\Omega$ or $\Gamma$. \\ 
\indent We will first prove an existence result for a linearization of \eqref{state_constraint_strong}, obtained by replacing the right-hand side of the equation with
\[
2k(\varphi) \beta(x,t) u_t+ f(x,t).
\]
We claim that under suitable regularity and smallness assumptions on the variable coefficients $\alpha$ and $\beta$, such a problem has a unique weak solution. Note that in the analysis of the state and adjoint problems below we allow the values of $\varphi$ to exceed the interval $[0,1]$. To make this possible, we extend the definition of the coefficients \eqref{interpolated_coefficients} here as follows:
\begin{equation} \label{interpolated_coefficients_extended}
\begin{aligned}
c^2=\,\begin{cases} c^2_l +\varphi(x)(c^2_f-c^2_l), \ & \varphi \in [0,1], \\
c^2_l, \ & \varphi<0, \\
c^2_f, \ & \varphi>1,
\end{cases}        
\end{aligned}
\end{equation}
and redefine analogously the sound diffusivity $b$ and coefficient $k$. \\
\indent On account of the compact embedding $H^1(\Omega) \hookrightarrow \hookrightarrow L^4(\Omega)$ and Ehrling's lemma \cite[Theorem 7.30]{renardy2006introduction}, we have
\begin{equation}\label{interpolationL4}
\begin{array}{l}
\|u\|_{L^{4}(\Omega)}\leqslant C(\epsilon)\|u\|_{L^{2}(\Omega)}+\epsilon\|\nabla u\|_{L^{2}(\Omega)}, \quad u \in H^1(\Omega),
\end{array}
\end{equation}
for any $\epsilon>0$, which we will frequently exploit in our energy arguments.
\begin{proposition} \label{Prop:LinWave}
	Given $T>0$, let $f \in L^2(0,T; L^2(\Omega))$, $g \in H^1(0,T; H^{-1/2}(\Gamma))$,  and 
	\[\alpha \in W^{1, \infty}(0,T; L^2(\Omega)) \cap L^\infty(0,T; L^\infty(\Omega)), \qquad \beta \in L^\infty(0,T; L^2(\Omega)). \]
	Further,  assume that the coefficient $\alpha$ satisfies the non-degeneracy condition \eqref{non-degeneracy} and that
	\begin{equation}
	(u_0, u_1) \in H^1(\Omega) \times H^1(\Omega).
	\end{equation}
	Then for every $\varphi \in L^1(\Omega)$,  there is a unique $u \in U$, which solves
	\begin{equation}\label{weakformorprob}
	\begin{aligned}
	\begin{multlined}[t]
	\intO	\alpha(t) u_{tt}(t) v \dx+\intO (c^2(\varphi) \nabla u(t)+b(\varphi)  \nabla u_t(t)) \cdot \nabla v \dx \\
	= \intO 2k(\varphi) \beta(t) u_t(t) v\dx+\intO f(t) v \dx+ \sca{g(t)}{v}_{\Gamma} \quad \text{a.e.\ in } (0,T),
	\end{multlined}
	\end{aligned}
	\end{equation}
	for all $v \in H^1(\Omega)$, with $(u, u_t)\vert_{t=0}=(u_0, u_1)$.	Furthermore, this solution satisfies the following energy estimate:
	\begin{equation} \label{energy_est_linear}
	\begin{aligned}
	\|u\|^2_{U} 
	\leqslant \,C_1 \exp(C_2T)(\|u_0\|_{H^1}^2+\|u_1\|^2_{H^1}+\nLtwoLtwo{f}^2+\|g\|_{H^1(H^{-1/2})}^2),
	\end{aligned}
	\end{equation}
	where $C_2=C_3(1+\|\alpha_t\|_{L^\infty(L^2)}+\|\beta\|_{L^\infty(L^2)}+\|\beta\|_{L^\infty(L^2)}^2)$.
\end{proposition}
\begin{proof}
	The proof can be rigorously conducted by a standard Faedo--Galerkin procedure; see, for example~\cite[Ch.\ XVIII, \S 5]{dautray2012mathematical} and~\cite[Ch.\ 7]{evans1998partial}. We present here only the derivation of the a priori bound. For simplicity, we denote the semi-discrete solution again by $u$. We comment further that we are going to use a generic positive constant $C$ for the rest of the proof and this constant might vary from line to line and can depend on $|\Omega|$, $c_{i},$ $b_{i}$, and $k_{i}$ for $i\in\{k,l\}$, as well as $\underline{\alpha}$ and $\overline{\alpha}.$   \\
	\indent Testing the discretized-in-space problem with $v={u_{t}(t)}$ and integrating in space yields the identity
	\begin{equation} \label{identity_linwellp}
	\begin{aligned}
	&\begin{multlined}[t]\frac12\frac{\textup{d}}{\textup{d}t}\|\sqrt{\alpha(t)}u_t(t)\|^2_{L^2(\Omega)}+\frac12\frac{\textup{d}}{\textup{d}t}\|c(\varphi)\nabla u(t)\|^2_{L^2(\Omega)}+\|\sqrt{b(\varphi)}\nabla u_t(t)\|^2_{L^2} \end{multlined}\\[1mm]
	=&\, (\frac12\alpha_t(t) u_t(t)+2k(\varphi) \beta(t) u_t(t)+f(t), u_{t}(t))_{L^2(\Omega)}+\sca{g(t)}{u_t(t)}_{\Gamma}
	\end{aligned}
	\end{equation}
	a.e.\ in time. We next integrate over $(0,t)$ and rely on the estimate
	\begin{equation}
	\begin{aligned}
	&\int_0^t \left| (\frac12\alpha_t(t) u_t(t)+2k(\varphi) \beta(t) u_t(t)+f(t), u_{t}(t))_{L^2(\Omega)}+\sca{g(t)}{u_t(t)}_{\Gamma}\right|\ds \\
	\leqslant&\, \begin{multlined}[t]\bigg(\frac12\nLinfLtwo{\alpha_t}+2\nLinf{k}\nLinfLtwo{\beta}\bigg)\nLtwotLfour{u_t}^2+ \nLtwoLtwo{f}\nLtwotLtwo{u_t}\\
	+C\|g\|_{L^2(H^{-1/2})}\|u_t\|_{L^2(H^1)}.\end{multlined}
	\end{aligned}
	\end{equation}
	By employing the compact embedding $H^1(\Omega) \hookrightarrow L^4(\Omega),$ Ehrling's lemma, and Young's inequality, from \eqref{identity_linwellp} we further obtain
	\begin{equation} \label{intermediate_est_linearization}
	\begin{aligned}
	&\begin{multlined}[t]\frac{\ulal}{2}\|u_t(t)\|^2_{L^2(\Omega)} \Big \vert_0^t+\frac{c_l^2}{2}\|\nabla u(t)\|^2_{L^2(\Omega)}\Big \vert_0^t+b_l\|\nabla u_t\|^2_{L^2(L^2)} \end{multlined}\\[1mm]
	\leqslant&\, \begin{multlined}[t]\epsilon\bigg(\frac12\nLinfLtwo{\alpha_t}+2k_f \nLinfLtwo{\beta}+1\bigg)\|\nabla u_t\|^2_{L^2(L^{2})}+ C\nLtwoLtwo{f}^2\\+C\left(\frac12\nLinfLtwo{\alpha_t}+2\nLinf{k}\nLinfLtwo{\beta}+1\right)\nLtwotLtwo{u_t}^2
\\	+C\|g\|_{L^2(H^{-1/2})}^2.\end{multlined}
	\end{aligned}
	\end{equation}
	Fixing a suitably small value of the parameter $\epsilon$ and using Gr\"onwall's lemma one furnishes the following from \eqref{intermediate_est_linearization}:
	\begin{equation}\label{afterGronweak}
	\begin{array}{ll}
	&\displaystyle\sup_{t \in (0,T)}\|u_{t}(t)\|^{2}_{L^{2}(\Omega)}+\sup_{t \in (0,T)}\|\nabla u(t)\|^{2}_{L^{2}(\Omega)}\\
	&\displaystyle\leqslant C\exp(C_{4}T)\left(\|(u_{0},u_{1})\|^{2}_{H^{1}\times L^{2}}+\nLtwoLtwo{f}^2+\|g\|_{L^2(H^{-1/2})}^2\right),
	\end{array}
	\end{equation}
	where $$C_{4}=C_{5}\left(1+\|\alpha_{t}\|_{L^{\infty}(L^{2})}+\|\beta\|_{L^{\infty}(L^{2})}\right),$$
	for some positive constant $C_{5}.$ Note that for $t\in(0,T),$ we can obtain a bound on $\nLtwo{u(t)},$ simply by employing 
	\begin{equation}\label{Linfest}
	\begin{array}{ll}
	\displaystyle\nLtwo{u(t)} \leqslant T \sup_{t \in (0,T)} \nLtwo{u_t(s)}\ds+\nLtwo{u_0}.
	\end{array}
	\end{equation}
	\noindent Additionally testing the discretized-in-space problem with $v=u_{tt}(t)$ gives
	\begin{equation} \label{identity2_linwellp}
	\begin{aligned}
	&\|\sqrt{\alpha(t)}u_{tt}(t)\|^2_{L^2(\Omega)}+\frac12\frac{\textup{d}}{\textup{d}t}\|\sqrt{b(\varphi)}\nabla u_t(t)\|^2_{L^2(\Omega)} \\[1mm]
	=&\, \begin{multlined}[t]-(c^2(\varphi)\nabla u(t), \nabla u_{tt}(t))_{L^2}+(2k(\varphi) \beta(t) u_{t}(t)+f(t), u_{tt}(t))_{L^2(\Omega)}\\+\sca{g(t)}{u_{tt}(t)}_{\Gamma}.\end{multlined}
	\end{aligned}
	\end{equation}
	After integration over $(0,t)$, we can rely on the following bound to estimate the $k$ and $f$ terms:
	\begin{equation}\label{comput1}
	\begin{aligned}
	&\intT (2k(\varphi) \beta(t) u_{t}(t)+f(t), u_{tt}(t))_{L^2(\Omega)} \ds \\
	\leqslant&\, \begin{multlined}[t] 2k_f \|\beta\|_{L^\infty(L^2)}\|u_t\|_{L^2(L^4)}\nLtwotLtwo{u_{tt}}+ \nLtwoLtwo{f}\nLtwotLtwo{u_{tt}}
	\end{multlined} \\
	\leqslant&\, \begin{multlined}[t] \frac{C}{\epsilon} k_f^2 \|\beta\|^2_{L^\infty(L^2)}\bigg(\|u_t\|^2_{L^2(L^2)}+\|\nabla u_{t}\|^{2}_{L^{2}(L^{2})}\bigg)+2\epsilon \nLtwotLtwo{u_{tt}}^2\\+\frac{1}{4 \epsilon}\nLtwoLtwo{f}^2, \end{multlined}
	\end{aligned}
	\end{equation}
	where we have again relied on the embedding $H^1(\Omega) \hookrightarrow L^4(\Omega)$. Next, we employ integration by parts with respect to time to obtain
	\begin{equation}\label{comput2}
	\begin{aligned}
	&\left|\intT (c^2(\varphi)\nabla u, \nabla u_{tt})_{L^2} \ds \right|\\
	=&\, \begin{multlined}[t] 
	\left|(c^2(\varphi)\nabla u(s), \nabla u_{t}(s))_{L^2} \Big \vert_0^t -\int_0^t (c^2(\varphi)\nabla u_t, \nabla u_{t})_{L^2} \ds \right|
	\end{multlined} \\
	\leqslant&\,\begin{multlined}[t]  \frac{1}{4 \epsilon} c_f^4 \nLtwo{\nabla u(t)}^2+\epsilon \nLtwo{\nabla u_t(t)}^2+\frac{1}{4} c_f^4 \nLtwo{\nabla u_0}^2+\nLtwo{\nabla u_1}^2  
	+c_f^2 \nLtwoLtwo{\nabla u_t}^2  \end{multlined}\\
	\leqslant&\,\begin{multlined}[t]  \frac{1}{4 \epsilon} c_f^4 (\sqrt{T}\nLtwoLtwo{\nabla u_t}+\nLtwo{\nabla u_0})^2+\epsilon \nLtwo{\nabla u_t(t)}^2+\frac{1}{4} c_f^4 \nLtwo{\nabla u_0}^2\\+\nLtwo{\nabla u_1}^2  +c_f^2 \nLtwoLtwo{\nabla u_t}^2. \end{multlined}
	\end{aligned}
	\end{equation}
	Note that $g \in H^1(0,T; H^{-1/2}(\Gamma)) \hookrightarrow C([0,T]; H^{-1/2}(\Gamma))$. Thus we have further
	\begin{equation}\label{comput3}
	\begin{aligned}
	&\left|\int_0^t \sca{g(t)}{u_{tt}(t)}_{\Gamma} \ds \right|\\
	=&\, \left|\sca{g(s)}{u_{t}(s)}_{\Gamma}\Big \vert_0^t-\int_0^t \sca{g_t}{u_{t}}_{\Gamma} \ds \right| \\
	\leqslant&\, \begin{multlined}[t]\frac{1}{4 \epsilon}\|g(t)\|^2_{H^{-1/2}}+\epsilon C\|u_t(t)\|_{H^1}^2+\frac{1}{4}\|g(0)\|^2_{H^{-1/2}}+\epsilon C\|u_1\|_{H^1}^2 \\
	+\frac{1}{4}\|g_t\|_{L^2(H^{-1/2})}^2+ C \|u_t\|_{L^2(H^1)}^2. \hphantom{\ldots}
	\end{multlined}
	\end{aligned}
	\end{equation}
	Combining \eqref{identity2_linwellp} with \eqref{comput1}, \eqref{comput2}, \eqref{comput3}, fixing a suitably small value of the positive parameter $\epsilon$ and further using the positive lower bounds of $\alpha$ and $b(\cdot)$, we furnish the following for a.e.\ $t\in(0,T):$
	\begin{equation}
	\begin{aligned}
	&\|u_{tt}\|^{2}_{L^{2}(L^{2})}+\|\nabla u_{t}(t)\|^{2}_{L^{2}}\\
	 \leqslant&\,\begin{multlined}[t] C(k^{2}_{f}\|\beta\|^{2}_{L^{\infty}(L^{2})}+1)\|u_{t}\|^{2}_{L^{2}(L^{2})}+C(k^{2}_{f}\|\beta\|^{2}_{L^{\infty}(L^{2})}+CT+1)\|\nabla u_{t}\|^{2}_{L^{2}(L^{2})}\\
	+C(\|u_{t}(t)\|^{2}_{L^{2}}+\|\nabla u_{t}(t)\|^{2}_{L^{2}})+C(\|f\|^{2}_{L^{2}(L^{2})}+\|g\|^{2}_{H^{1}(H^{-1/2})}+\|(u_{0},u_{1})\|^{2}_{H^{1}}). \end{multlined}
	\end{aligned}
	\end{equation}
	Now incorporating \eqref{Linfest} in the above inequality, adding the resulting expression to \eqref{afterGronweak}, and once again using Gr\"onwall's inequality, we render the desired estimate \eqref{energy_est_linear}. This finishes the proof of Proposition \ref{Prop:LinWave}.
	
\end{proof}
\noindent We can now prove a well-posedness result for the semilinear problem by relying on Banach's fixed-point theorem.
\begin{theorem} \label{Thm:Wellposedness}
	Given $T>0$, let $g \in H^1(0,T; H^{-1/2}(\Gamma))$,  and 
	\[\alpha \in W^{1, \infty}(0,T; L^2(\Omega)) \cap L^\infty(0,T; L^\infty(\Omega)). \]
	Further,  assume that $\alpha$ satisfies the non-degeneracy condition \eqref{non-degeneracy} and that
	$(u_0, u_1) \in H^1(\Omega) \times H^1(\Omega)$. Then there exists $\delta>0$, such that if
	\begin{equation} \label{condition_data}
	\|u_0\|^2_{H^1}+\|u_1\|^2_{H^1}+\|g\|^2_{H^1(H^{-1/2})} \leqslant \delta,
	\end{equation}
	then for every $\varphi \in L^1(\Omega)$,  there is a unique $u \in U$, such that
	\begin{equation}\label{weakformmainprob} 
	\begin{aligned}
	\begin{multlined}[t]
	\intO	\alpha(t) u_{tt}(t) v \dx+\intO (c^2(\varphi) \nabla u(t)+b(\varphi) \nabla u_t(t)) \cdot \nabla v \dx \\
	= \intO 2k(\varphi) u_t^2(t) v\dx+ \int_{\Gamma} g(t) v \dx \quad \text{a.e.\ in } (0,T),
	\end{multlined}
	\end{aligned}
	\end{equation}
	for all $v \in H^1(\Omega)$, with $(u, u_t)\vert_{t=0}=(u_0, u_1)$.	Furthermore, the solution satisfies the following energy bound:
	\begin{equation} \label{Energy_bound_nl}
	\begin{aligned}
	\|u\|^2_{U} \leqslant \,C_1 \exp(C_2T)(\|u_0\|_{H^1}^2+\|u_1\|^2_{H^1}+\|g\|_{H^1(H^{-1/2})}^2).
	\end{aligned}
	\end{equation}
\end{theorem}
\begin{proof}
	The statement follows by Banach's fixed-point theorem applied to the mapping\[
	\mathcal{T}: w \mapsto u,
	\]
where $w$ is taken from the ball  
	\begin{equation}
	\begin{aligned}
	B=\left\{w \in U\, :\ \|w\|_{U} \leqslant M, \ (w(0), w_t(0))=(u_0, u_1)\right\},
	\end{aligned}
	\end{equation}	
	with $M>0$ , and  $u$ solves the linear problem
	\begin{equation} 
	\begin{aligned}
	\begin{multlined}[t]
	\intO	\alpha(t) u_{tt}(t) v \dx+\intO (c^2 \nabla u(t)+b \nabla u_t(t)) \cdot \nabla v \dx \\
	= \intO 2k w_t(t) u_t(t) v\dx+ \int_{\Gamma} g(t) v \dx \quad \text{a.e.\ in } (0,T),
	\end{multlined}
	\end{aligned}
	\end{equation}
	for all $v \in H^1(\Omega)$, with $(u, u_t)\vert_{t=0}=(u_0, u_1)$.  Note that $\beta=w_t$ fulfills
	\[
	\|\beta\|_{L^\infty(L^2)}= \|w_t\|_{L^\infty(L^2)} \leqslant CM.
	\]
	Then all the assumptions of Proposition~\ref{Prop:LinWave} hold, so the mapping $\mathcal{T}$ is well-defined. The set $B$ is non-empty and closed with respect to the metric induced by the norm $\|\cdot \|_{U}$. Furthermore, on account of the proven energy estimate \eqref{energy_est_linear} with $f=0$, we have
	\begin{equation}
	\begin{aligned}
	\|u\|^2_{U}  \leqslant C_1\exp(C_2(M) T) \delta.
	\end{aligned}
	\end{equation}
	Thus, choosing $\delta>0$ so that
	\begin{equation} \label{condition1_delta}
	\sqrt{C_1\exp(C_2(M) T) \delta} \leqslant  M \end{equation}
	implies 
	$\mathcal{T}(B) \subset B$. \\
	\indent We next prove strict contractivity of the mapping. To this end, let $w^{(1)}$, $w^{(2)} \in B$. We denote $u^{(1)}=\mathcal{T} w^{(1)}$ and $u^{(2)}=\mathcal{T} w^{(2)}$. Further, let $\overline{w}=w^{(1)}-w^{(2)}$. The difference $\overline{u}=u^{(1)}-u^{(2)}$ then satisfies
	\begin{equation} 
	\begin{aligned}
	& \intO	\alpha(t) \overline{u}_{tt}(t) v \dx+\intO (c^2(\varphi) \nabla \overline{u}(t)+b(\varphi) \nabla \overline{u}_t(t)) \cdot \nabla v \dx \\
	=&\, \intO 2k(\varphi)( \overline{u}_tw_t^{(1)}+u_t^{(2)}\overline{w}_t) v\dx
	\end{aligned}
	\end{equation}
	a.e. in time, and has zero initial conditions. The energy bound  \eqref{energy_est_linear} with $u_0=u_1=0$, $g=0$, and
	\[
	f= 2k(\varphi)\overline{w}_t, \quad \beta = w_t^{(1)}, 
	\]
	immediately yields
	\begin{equation}
	\begin{aligned}
	\begin{multlined}[t]\|\overline{u}\|^2_{U} \end{multlined}
	\leqslant& \,C_1 \exp(C_2T)\nLtwoLtwo{2k(\varphi)u_t^{(2)}\, \overline{w}_t }^2 \\
	\leqslant& \,C_1 \exp(C_2T)2 k_f \|u_t^{(2)}\|^2_{L^2(L^4)} \|\overline{w}_t \|^2_{L^\infty(L^4)}\\
	\leqslant&\, \,C_3 \exp(C_2T)T\delta \|\overline{w}\|^2_{U}.
	\end{aligned}
	\end{equation}
	Thus by additionally reducing $\delta$, we obtain strict contractivity of the mapping $\mathcal{T}$ with respect to $\|\cdot\|_{U}$. Banach's fixed-point theorem thus provides us with a unique solution $p= \mathcal{T}(p)$ of the problem in $B$.
\end{proof}
\section{Formulation as an optimal control problem} \label{Sec:OptimalControlFormulation}
We can restate the phase-field problem as an optimal control problem by introducing the mapping
\[ S: L^1(\Omega) \ni \varphi \mapsto u \in U, \]
which is well-defined under the assumptions of Theorem~\ref{Thm:Wellposedness}. Our next goal is to show differentiability of this control-to-state operator. To this end, we first prove a stability result with respect to a lower-order norm (in time). To formulate it, let us introduce the space
\begin{equation}
\begin{aligned}
X = \{u \in L^\infty(0,T; H^1(\Omega)): \ u_t \in  L^2(0,T; H^1(\Omega))\}
\end{aligned}
\end{equation}
with the associated norm
\begin{equation} \label{norm_X}
\begin{aligned}
\|u\|_{X}=&\, \begin{multlined}[t] \left\{\sup_{t \in (0,T)}\nHone{u(t)}^2 +\int_0^T\| u_{t}(s)\|_{H^1}^2 \ds\right\}^{1/2}.\end{multlined}
\end{aligned}
\end{equation}
\begin{proposition}\label{Propstability}
Under the assumptions of Theorem~\ref{Thm:Wellposedness}, let $u^{(1)}$ and $u^{(2)}$ be the weak solutions of the state problem corresponding to the phase-field functions $\varphi_1$ and $\varphi_2$, respectively. Then 
	\begin{equation}
	\begin{aligned}
\|\overline{u}\|_{X}  \leqslant C(\delta, T)\|\overline{\varphi}\|_{L^\infty}
	\end{aligned}
	\end{equation}
a.e.\ in time, where $\overline{\varphi}=\varphi_1-\varphi_2$ and $\overline{u}=u^{(1)}-u^{(2)}$.	
\end{proposition}
\begin{proof}
We begin the proof by noting that the difference $\overline{u}=u^{(1)}-u^{(2)}$ satisfies
\begin{equation} 
\begin{aligned}
& \intO	\alpha(t) \overline{u}_{tt}(t) v \dx+\intO (c^2(\varphi_1) \nabla \overline{u}(t)+b(\varphi_1) \nabla \overline{u}_t(t)) \cdot \nabla v \dx \\
=&\,\begin{multlined}[t] \intO -\overline{\varphi}((c^2_f-c^2_l)\nabla u^{(2)}(t)+ (b_f-b_l)\nabla u_t^{(2)}(t)) \cdot \nabla v\dx \\
+ \intO (2k(\varphi_1) \overline{u}_t(t)(u_t^{(1)}(t)+u_t^{(2)}(t))+2 \overline{\varphi}(k_f-k_l)(u_t^{(2)}(t))^2)v \dx
\end{multlined}
\end{aligned}
\end{equation}
for all $v \in H^1(\Omega)$ a.e.\ in time, and has zero initial conditions. Testing with $v=\overline{u}_t(t) \in H^1(\Omega)$ at first results in the identity
\begin{equation} \label{identity_continuity}
\begin{aligned}
&\begin{multlined}[t]\frac12\frac{\textup{d}}{\textup{d}t}\|\sqrt{\alpha(t)}\overline{u}_t(t)\|^2_{L^2(\Omega)}+\frac12\frac{\textup{d}}{\textup{d}t}\|c(\varphi_1)\nabla \overline{u}(t)\|^2_{L^2(\Omega)}\\
+\|\sqrt {b(\varphi_1)}\nabla \overline{u}_t\|^2_{L^2} \end{multlined}\\[1mm]
=&\,\begin{multlined}[t] (\frac12\alpha_t \overline{u}_t, \overline{u}_{t})_{L^2(\Omega)}-(\overline{\varphi}((c^2_f-c^2_l)\nabla u^{(2)}+ (b_f-b_l)\nabla u_t^{(2)}), \nabla \overline{u}_t)_{L^2} \\
+ ( 2k(\varphi_1) \overline{u}_t(u_t^{(1)}+u_t^{(2)})+2 \overline{\varphi}(k_f-k_l)(u_t^{(2)})^2,\overline{u}_t )_{L^2}.\end{multlined}
\end{aligned}
\end{equation}
We focus on estimating the last two terms since the rest can be handled as in the proof of Proposition~\ref{Prop:LinWave}. After integration over $(0,t)$, we have
\begin{equation} \label{est1_continuity}
\begin{aligned}
&-\inttO (\overline{\varphi}((c^2_f-c^2_l)\nabla u^{(2)}+ (b_f-b_l)\nabla u_t^{(2)}), \nabla \overline{u}_t)_{L^2} \dxs \\
\leqslant& \,C\nLinf{\overline{\varphi}}(\nLtwotLtwo{\nabla u^{(2)}}+\nLtwotLtwo{\nabla u_t^{(2)}})\nLtwotLtwo{ \nabla \overline{u}_t} \\
\leqslant&\, C(T, \delta)  \nLinf{\overline{\varphi}}^2+\epsilon \|\nabla \overline{u}_t\|^2_{L^2(L^2)},
\end{aligned}
\end{equation}
on account of $u^{(2)}$ satisfying the energy bound \eqref{Energy_bound_nl}. Similarly, 
\begin{equation} \label{est2_continuity}
\begin{aligned}
& \inttO ( (2k(\varphi_1) \overline{u}_t(u_t^{(1)}+u_t^{(2)})+2 \overline{\varphi}(k_f-k_l)(u_t^{(2)})^2),\overline{u}_t )_{L^2} \dxs \\
\leqslant& \begin{multlined}[t]\,2k_f \|\overline{u}_t\|_{L^2(L^4)}^2(\|u^{(1)}_t\|_{L^\infty(L^2)}+\|u^{(2)}_t\|_{L^\infty(L^2)})\\ +2\nLinf{\overline{\varphi}}\|u_t^{(2)}\|_{L^2(L^4)} \|u_t^{(2)}\|_{L^\infty(L^2)}\|\overline{u}_t\|_{L^2(L^4)} \end{multlined} \\
\leqslant&\, C \sqrt{\delta} \|\overline{u}_t\|_{L^2(L^4)}^2+C(T, \delta)  \nLinf{\overline{\varphi}}^2+\epsilon \|\overline{u}_t\|^2_{L^2(L^4)}\\
\leqslant&\, C \sqrt{\delta} \left(C(\epsilon)\|\overline{u}_t\|_{L^2(L^2)}+ \epsilon\|\nabla \overline{u}_t\|_{L^2(L^2)}\right)^2+C(T, \delta)  \nLinf{\overline{\varphi}}^2+\epsilon \|\overline{u}_t\|^2_{L^2(L^4)},
\end{aligned}
\end{equation}
where we have relied on Ehrling's lemma in the last line. Thus the desired estimate follows from \eqref{identity_continuity}--\eqref{est2_continuity} by standard arguments.
\end{proof}
Before stating the differentiability of the control-to-state operator, we prove the following preparatory result for $S'(\varphi) h=u^*$. 
\begin{proposition}\label{Prop:u_star}
Under the assumptions of Theorem~\ref{Thm:Wellposedness} and for a fixed $h\in L^{\infty}(\Omega)\cap H^{1}(\Omega),$ there exists a unique $u^* \in X$ which solves	
\begin{equation}\label{diff_state}
\begin{aligned}
&\begin{multlined}[t]
\langle \alpha(t) u^*_{tt}(t), v \rangle_{\Omega}+\intO (c^2(\varphi) \nabla u^*(t)+b(\varphi) \nabla u^*_t(t)) \cdot \nabla v \dx \end{multlined}\\
=&\, \begin{multlined}[t]-\intO (2c(\varphi)c'(\varphi)h \nabla u(t)+b'(\varphi)h \nabla u_t(t)) \cdot \nabla v \dx \\+\intO 4k(\varphi) u_t u_t^*(t) v\dx+\intO 2k'(\varphi)h u_t^2(t) v\dx  \end{multlined}
\end{aligned}
\end{equation}
a.e. in time, with $(u^*, u^*_t)\vert_{t=0}=(0, 0)$.
\end{proposition}
\begin{proof}
	We only present the uniform a priori estimates of $u^{*}$ towards the proof of Proposition~\ref{Prop:u_star}. A rigorous proof can be performed by employing standard Faedo--Galerkin type approximation argument.\\
	\indent	In the direction of obtaining a priori estimates we use the test function $v=u^{*}_{t}(t)$ in the weak formulation \eqref{diff_state} and furnish the following:
	\begin{equation}\label{afttestut*}
	\begin{aligned}
	&\int_{\Omega}\alpha u^{*}_{tt}u^{*}_{t}\dx+\int_{\Omega}c^{2}(\varphi)\nabla u^{*}\cdot\nabla u^{*}_{t}\dx+\int_{\Omega}b(\varphi)\nabla u^{*}_{t}\cdot \nabla u^{*}_{t}\dx\\
	 =& \, \begin{multlined}[t]-2\int_{\Omega}c(\varphi) c'(\varphi)h\nabla u\cdot\nabla u^{*}_{t}\dx-\int_{\Omega}b'(\varphi)h\nabla u_{t}\cdot\nabla u^{*}_{t}\dx\\+4\int_{\Omega}k(\varphi)u_{t}|u^{*}_{t}|^{2}\dx
	 +2\int_{\Omega}k'(\varphi)h u^{2}_{t}u^{*}_{t}\dx. \end{multlined}
	\end{aligned}
	\end{equation}
	The equality \eqref{afttestut*} can be written as
	\begin{equation}\label{reformafttestut}
	\begin{aligned}
	&\frac{1}{2}\ddt\int_{\Omega}\alpha|u^{*}_{t}|^{2}\dx+\frac{1}{2}\ddt\int_{\Omega}c^{2}(\varphi)|\nabla u^{*}|^{2}\dx+\int_{\Omega}b(\varphi)|\nabla u^{*}_{t}|^{2}\dx\\
	=&\,\begin{multlined}[t]-2\int_{\Omega}c(\varphi) c'(\varphi)h\nabla u\cdot\nabla u^{*}_{t}-\int_{\Omega}b'(\varphi)h\nabla u_{t}\cdot\nabla u^{*}_{t}\dx\\+4\int_{\Omega}k(\varphi)u_{t}|u^{*}_{t}|^{2}\dx
	+2\int_{\Omega}k'(\varphi)h u^{2}_{t}u^{*}_{t}\dx+\frac{1}{2}\int_{\Omega}\alpha_{t}|u^{*}_{t}|^{2}\dx=\sum_{i=1}^{5} I_{i}.\end{multlined}
	\end{aligned}
	\end{equation}
	In the following we will estimate the terms $I_{i}.$ The generic constants $C$ in the following inequalities might differ from one line to another and they can depend on $\|h\|_{L^{\infty}(\Omega)},$ $\|(u_{0},u_{1})\|_{H^{1}(\Omega)\times H^{1}(\Omega)},$ $\|g\|_{H^1(H^{-1/2})},$ $\underline{\alpha},$ $\overline{\alpha},$ $c_{i},$ $k_{i}$, and $b_{i}$ for $i\in\{f,l\}.$
	\begin{equation}\label{I1}
	\begin{array}{ll}
	|I_{1}|\leqslant C\|\nabla u\|^{2}_{L^{2}}+\epsilon\|\sqrt{b(\varphi)}\nabla u^{*}_{t}\|^{2}_{L^{2}},
	\end{array}
	\end{equation}
	for some positive parameter $\epsilon,$ since $\|2c(\varphi)c'(\varphi)\|_{L^{\infty}}\leqslant C,$ which follows by performing the derivative of \eqref{interpolated_coefficients}$_{1}.$\\
	\indent Now let us estimate $I_{2}:$
	\begin{equation}\label{I2}
	\begin{array}{ll}
	|I_{2}|\leqslant C\|\nabla u_{t}\|^{2}_{L^{2}}+\epsilon\|\sqrt{b(\varphi)}\nabla u^{*}_{t}\|^{2}_{L^{2}},
	\end{array}
	\end{equation}
	for some positive parameter $\epsilon.$ The term $I_{3}$ can be estimated as follows:
	\begin{equation}\label{I3}
	\begin{aligned}
	|I_{3}|&\displaystyle\leqslant C\|u_{t}\|_{L^{4}}\|u^{*}_{t}\|_{L^{4}}\|u^{*}_{t}\|_{L^{2}}\\
	&\displaystyle \leqslant C\|u_{t}\|^{2}_{L^{4}}\|u^{*}_{t}\|^{2}_{L^{2}}+\epsilon\|u^{*}_{t}\|^{2}_{L^{4}}\\
	&\displaystyle \leqslant C\|u_{t}\|^{2}_{L^{4}}\|\sqrt{\alpha}u^{*}_{t}\|^{2}_{L^{2}}+\epsilon\bigg(\|u^{*}_{t}\|^{2}_{L^{2}}+\|\nabla u^{*}_{t}\|^{2}_{L^{2}}\bigg)\\
	&\displaystyle \leqslant C\left(\left(\|u_{t}\|^{2}_{L^{4}}+1\right)\|\sqrt{\alpha}u^{*}_{t}\|^{2}_{L^{2}}+\epsilon \|\sqrt{b(\varphi)}\nabla u^{*}_{t}\|^{2}_{L^{2}}\right).
	\end{aligned}
	\end{equation} 	
	Finally, $I_{4}$ and $I_{5}$ admit the following estimates:
	\begin{equation}\label{I4}
	\begin{array}{l}
	\displaystyle|I_{4}|\leqslant C\|u^{*}_{t}\|^{2}_{L^{2}}+\|u_{t}\|^{4}_{L^{4}}
	\end{array}
	\end{equation}	
	and 
	\begin{equation}\label{I5}
	\begin{aligned}
	|I_{5}|&\displaystyle\leqslant C\|\alpha_{t}\|_{L^{2}}\|u^{*}_{t}\|^{2}_{L^{4}}\\
	&\displaystyle \leqslant C\|\alpha_{t}\|_{L^{2}}\|u^{*}_{t}\|^{2}_{L^{2}}+\epsilon\|\alpha_{t}\|_{L^{2}}\|\nabla u^{*}_{t}\|^{2}_{L^{2}}\\
	&\displaystyle \leqslant C\left(\|\alpha_{t}\|_{L^{2}}\|\sqrt{\alpha}u^{*}_{t}\|^{2}_{L^{2}}+\epsilon\|\alpha_{t}\|_{L^{2}}\|\sqrt{b(\varphi)}\nabla u^{*}_{t}\|^{2}_{L^{2}}\right).
	\end{aligned}
	\end{equation}	
	In view of the above estimates of $I_{i},$ $i\in\{1,...,5\},$ the fact that $u\in U,$ $\alpha_{t}\in  L^{\infty}(0,T;L^{2})$ and the functions $\alpha,$ $c^{2}(\varphi),$ and $b(\varphi)$ are bounded from below by positive constants, one can use \eqref{reformafttestut} and Gr\"{o}nwall's inequality to furnish the following for small enough values of $\epsilon>0$:
	\begin{equation}\label{boundsu*}
	\begin{array}{ll}
	\displaystyle\|u^{*}_{t}\|_{L^{\infty}(0,T;L^{2})}+\|\nabla u^{*}\|_{L^{\infty}(0,T;L^{2})}+\|\nabla u^{*}_{t}\|_{L^{2}(0,T;L^{2})}\leqslant C,
	\end{array}
	\end{equation}
	which completes the proof.
\end{proof}
\noindent We are now ready to prove differentiability of the control-to-state operator. 
\begin{theorem}\label{Thm:Continuity} \label{Thm:ParameterToState}
Under the assumptions of Theorem~\ref{Thm:Wellposedness}, the control-to-state operator $S: \Phi_{\textup{ad}} \rightarrow U$ is well-defined. Furthermore, it is Fr\'echet differentiable as an operator $S: \Phi_{\textup{ad}} \rightarrow X$ and its directional derivative at $\varphi \in L^\infty(\Omega)$ in the direction of $h \in L^\infty(\Omega)$ is given by
	\begin{equation}
	S'(\varphi) h=u^*, 
	\end{equation}
where $u^*$ is the unique solution of \eqref{diff_state}.
\end{theorem}
\begin{proof}
Given $h \in L^\infty(\Omega)$, let \[r=S(\varphi+h)-S(\varphi)-u^*.\]
If we denote $u^{h}=S(\varphi+h)$ and $u=S(\varphi)$, then $r$ solves the following problem:	
\begin{equation}\label{rv}
\begin{aligned}
& \intO	\alpha r_{tt} v \dx+\intO (c^2(\varphi) \nabla r+b(\varphi) \nabla r_t) \cdot \nabla v \dx \\
=&\,\begin{multlined}[t] \intO -\left\{(c^2(\varphi+h)-c^2(\varphi)-2c(\varphi)c'(\varphi)h)\nabla u \right.\\ \left.+ (b(\varphi+h)-b(\varphi)-b'(\varphi)h)\nabla u_t \right\}  \cdot \nabla v\dx \\
-\intO (c^2(\varphi+h)-c^2(h)) \nabla (u^h-u)\\+(b(\varphi+h)-b(\varphi)) \nabla (u^h_t-u_t) )\cdot \nabla v \dx\\
+\intO 2(k(\varphi+h)-k(\varphi)) ((u_t^h)^2-u_t^2) v\dx -\intO 2k(\varphi)(u_t^h-u_t)^2)v \dx \\
+ \intO 2(k(\varphi+h)-k(\varphi)-k'(\varphi)h)u_t^2)v \dx +\intO4k(\varphi) u_t r_t v\dx=\sum_{i=1}^{4}J_{i}\\
\end{multlined}
\end{aligned}
\end{equation}
for all test functions $v \in H^{1}(\Omega)$, a.e.\ in time, supplemented by zero initial conditions. We wish to show that
	 \[\|r\|_X =o(\nLinf{h}) \quad \text{ as }\ \nLinf{h} \rightarrow 0.\]
The proof follows by choosing $v=r_t(t)$ in \eqref{rv}, which belongs to $H^1(\Omega)$ on account of Theorem~\ref{Thm:Wellposedness} and Proposition~\ref{Prop:u_star}. Consequently from \eqref{rv} (after integrating in time) one has for a.e.\ $t\in(0,t):$
\begin{equation}\label{fromrv}
\begin{array}{ll}
&\displaystyle\frac{1}{2}\int_{\Omega}\alpha|r_{t}|^{2}(t)\dx+\frac{1}{2}\int_{\Omega}c^{2}(\varphi)|\nabla r|^{2}(t)\dx+\int_{0}^{t}\int_{\Omega}b(\varphi)|\nabla r_{t}|^{2}\dxs\\
&\displaystyle=\sum_{i=1}^{4}\int_{0}^{t} J_{i}\ds+\frac{1}{2}\int_{0}^{t}\int_{\Omega}\alpha_{t}|r_{t}|^{2}\dxs.
\end{array}
\end{equation}
We can then rely on the following estimates for the terms appearing in the right hand side of \eqref{fromrv}
\begin{equation}
\begin{aligned}
& \bigg|\int_{0}^{t}J_{1}\ds\bigg|=\bigg|\begin{multlined}[t] \inttO -\left\{(c^2(\varphi+h)-c^2(\varphi)-2c(\varphi)c'(\varphi)h)\nabla u \right.\\ \left.+ (b(\varphi+h)-b(\varphi)-b'(\varphi)h)\nabla u_t \right\}  \cdot \nabla r_t \dxs \end{multlined}\bigg| \\
\leqslant&\, C \|h\|_{L^\infty}^2 (\nLtwotLtwo{\nabla u}+\nLtwotLtwo{\nabla u_t})\nLtwotLtwo{\nabla r_t} \\
\leqslant&\, C(\delta, T)\|h\|_{L^\infty}^2\nLtwotLtwo{\sqrt{b(\varphi)}\nabla r_t},
\end{aligned}
\end{equation}	 
where in the last line we have used energy bound \eqref{Energy_bound_nl}. Further, we have
\begin{equation}
\begin{aligned}
&\bigg|\int_{0}^{t}J_{2}\ds\bigg|\\
&=\begin{multlined}[t]
\bigg|\inttO ((c^2(\varphi+h)-c^2(h)) \nabla (u^h-u)\\+(b(\varphi+h)-b(\varphi)) \nabla (u^h_t-u_t) )\cdot \nabla r_t \dxs\bigg|
\end{multlined} \\
&\leqslant\, C \|h\|_{L^\infty} (\nLtwotLtwo{\nabla (u^h-u)}+\nLtwotLtwo{\nabla (u_t^h-u)})\nLtwotLtwo{\nabla r_t}\\
&\leqslant\, C(\delta, T) \|h\|^2_{L^\infty}\nLtwotLtwo{\sqrt{b(\varphi)}\nabla r_t},
\end{aligned}
\end{equation}
on account of continuity of the acoustic pressure with respect to $\varphi$, obtained in Proposition \ref{Propstability}. The $k$ terms can be handled in a similar manner: 
\begin{equation}
\begin{aligned}
&\bigg|\int_{0}^{t}J_{3}\ds\bigg|\\
&=\inttO 2(k(\varphi+h)-k(\varphi)) ((u_t^h)^2-u_t^2) r_{t}\dxs -\inttO 2k(\varphi)(u_t^h-u_t)^2)r_{t} \dxs \\
& \leqslant\, \begin{multlined}[t]
C\|h\|_{L^\infty}\|u^h_t-u_t\|_{L^\infty(L^2)}(\|u^h_t\|_{L^\infty(L^2)}+\|u_t\|_{L^\infty(L^2)})\|r_t\|_{L^2(L^4)} \\
+C \|u^h_t-u_t\|_{L^\infty(L^2)}\|u^h_t-u_t\|_{L^2(L^4)}\|r_t\|_{L^2(L^4)}.
\end{multlined}\\
&\leqslant \, C(\delta,T) \|h\|^{2}_{L^{\infty}}+C(\epsilon)\|\sqrt{\alpha}r_{t}\|^{2}_{L^{2}(L^{2})}+\epsilon\|\sqrt{b(\varphi)}\nabla r_{t}\|^{2}_{L^{2}(L^{2})},
\end{aligned}
\end{equation}
where we have used the energy bound \eqref{Energy_bound_nl}, the stability estimate from Proposition \ref{Propstability} to bound the term $\|u^{h}_{t}-u_{t}\|_{L^{2}(L^{4})}$ by $C(\delta,T)\|h\|_{L^{\infty}},$ Young's inequality and Ehrling's lemma.
Further,
\begin{equation}
\begin{aligned}
&\bigg|\int_{0}^{t}J_{4}\ds\bigg|\\
&=\begin{multlined}[t] \intO 2(k(\varphi+h)-k(\varphi)-k'(\varphi)h)u_t^2)r_t \dx +\intO4k(\varphi) u_t r_t^2\dx \end{multlined} \\
&\leqslant \, C \|h\|^2_{L^\infty}\|u_t\|_{L^\infty(L^2)}\|u_t\|_{L^2(L^4)}\|r_t\|_{L^2(L^4)}+ 4k_f \|u_t\|_{L^\infty(L^2)}\|r_t\|^2_{L^2(L^4)}\\
&\leqslant \, C(\delta,T)\bigg(\|h\|^{4}_{L^{\infty}}+C(\epsilon)\|\sqrt{\alpha}r_{t}\|^{2}_{L^{2}(L^{2})}+\epsilon\|\sqrt{b(\varphi)}\nabla r_{t}\|^{2}_{L^{2}(L^{2})}\bigg).
\end{aligned}
\end{equation}
The final term in the right hand side of \eqref{fromrv} can be estimated as
\begin{equation}\label{finalrv}
\begin{array}{ll}
&\displaystyle \bigg|\frac{1}{2}\int_{0}^{t}\int_{\Omega}\alpha_{t}|r_{t}|^{2}\dxs\bigg|\\
&\displaystyle \leqslant C\left(\|\alpha_{t}\|_{L^{2}}\|\sqrt{\alpha} r_{t}\|^{2}_{L^{2}}+\epsilon\|\alpha_{t}\|_{L^{2}}\|\sqrt{b(\varphi)}\nabla r_{t}\|^{2}_{L^{2}}\right)
\end{array}
\end{equation}
Using the above estimates in \eqref{fromrv} and implementing Gronwall inequality lead to the bound
\[
\|r\|_X \leqslant C \|h\|^2_{L^\infty},
\]
which completes the proof.
\end{proof}
\section{Existence of a minimizer} \label{Sec:ExistenceMinimizer}
We next wish to prove that the minimization problem has a solution and to derive the necessary optimality conditions of first order.
\subsection{Existence of a minimizer}
We begin by proving the existence of a minimizer.
\begin{proposition}
Under the assumptions of Theorem~\ref{Thm:Wellposedness}, the problem \eqref{objective}--\eqref{state_constraint_strong} has a minimizer.
\end{proposition}
\begin{proof}
We first introduce the feasible set as
\begin{equation}
\begin{aligned}
\mathcal{F}_{\textup{ad}} = \left\{(u, \varphi) \in U \times\Phi_{\textup{ad}}: (u, \varphi) \ \ \textup{solves the state problem}\,  \right\};
\end{aligned}
\end{equation}	
cf.~\eqref{admissible_sets}. Since $J^\varepsilon$ is bounded from below, there exists a minimizing sequence $\{(u_j, \varphi_j)\}_{j \in \N} \subset U \times \Phi_{\textup{ad}}$, such that
	\[
	\lim_{j \rightarrow \infty} J^\varepsilon(u_j, \varphi_j)=\inf_{(u, \varphi) \in \mathcal{F}_{\textup{ad}}} J^\varepsilon(u, \varphi).
	\]
	The uniform boundedness of $\left(J^\varepsilon(u_j, \varphi_j)\right)_{j \in \N}$ implies a uniform bound for $\|\nabla \varphi_j\|_{L^2}$. Since $\varphi_{j} \in \Phi_{\textup{ad}}$, we also have that $\|\varphi_j\|_{L^\infty} \leqslant 1$ for all $j \in \N$. Furthermore, Theorem~\ref{Thm:Wellposedness} implies that the sequence $\{u_j\}_{j \in \N}  \subset U$ is bounded as well. Hence there exists $(\overline{u}, \overline{\varphi})$ and a subsequence that we do not relabel, such that
	\begin{equation}
	\begin{aligned}
	u_j &\rightarrow \overline{u} \ &&\text{ weakly in} \ H^1(0,T; H^1(\Omega)), \\
	u_{j,tt} &\rightarrow \overline{u}_{tt} \ &&\text{ weakly in} \ L^2(0,T; L^2(\Omega)), \\
	\varphi_j &\rightarrow \overline{\varphi} \ &&\text{ weakly in} \ H^1(\Omega), \\
		\varphi_j &\rightarrow \overline{\varphi} \ &&\text{ strongly in} \ L^2(\Omega).
	\end{aligned}
	\end{equation}
From the weak lower semi-continuity property of convex functionals and strong convergence of $\{\varphi_j\}_{j \in \N}$ in $L^2(\Omega)$, it follows that
\[
J^\varepsilon(\overline{u}, \overline{\varphi}) \leqslant \lim_{j \rightarrow \infty} J^\varepsilon(u_j, \varphi_j).
\]	
Now in order to show that the couple $(\overline{u},\overline{\varphi})$ solves \eqref{state_constraint_strong} in a weak sense we will pass to the limit as $j \rightarrow \infty$ in 
\begin{equation} 
\begin{aligned}
\begin{multlined}[t]
\intO \alpha(t) u_{j,tt}(t) v \dx +\intO (c^2(\varphi_j) \nabla u_j(t)+b(\varphi_j) \nabla u_{j, t}(t)) \cdot \nabla v \dx \\
= \intO 2k(\varphi_j) u_{j,t}^2(t) v\dx+\sca{g(t)}{v}_{\Gamma},
\end{multlined}
\end{aligned}
\end{equation}
for a.e.\ $t\in(0,T)$ and with $v\in H^{1}(\Omega).$ We will justify this passage here in the relatively complicated term $\displaystyle \intO 2k(\varphi_j) u_{j,t}^2(t) v\dx.$ Passage of the limit in the other summands can be similarly done by using Lebesgue dominated convergence theorem.\\
\indent Since $\{u_{j}\}_{j}$ lies in a bounded set of \[W=L^{2}(0,T;H^{1}(\Omega))\cap H^{1}(0,T;L^{2}(\Omega))\] and $W$ is compactly embedded into $L^{p}(\Omega\times(0,T))$ for some $p>2,$ hence up to a non-relabeled subsequence 
$$k(\varphi_{j})u^{2}_{j,t} \rightarrow  k(\varphi)u_{t}^2\ \ \mbox{a.e.}$$
Now the  generalized Lebesgue theorem furnishes that
 $$k(\varphi_{j})u^{2}_{j,t} \rightarrow  k(\varphi)u_{t}^2\,\,\mbox{in}\,\, L^{1}(0,T;L^{3}(\Omega)),$$
 since $k(\varphi_{j})$ is uniformly bounded in $L^{\infty}(\Omega),$ we know that $u_{j,t}$ is uniformly bounded in $L^{2}(0,T;L^{6}(\Omega)).$ 
 Hence we can pass to the limit in the term \[\displaystyle\int_{0}^{T}\intO 2k(\varphi_j) u_{j,t}^2 v(x)h(t)\dx\dt,\] for all $v\in H^{1}(\Omega)$ and all smooth functions $h=h(t).$ A density argument furnishes the required convergence
 $$\displaystyle \intO 2k(\varphi_j) u_{j,t}^2(t) v\dx \rightarrow  \intO 2k(\varphi) u_{t}^2(t) v\dx,$$
 for a.e.\ $t\in(0,T)$ and with $v\in H^{1}(\Omega).$
 As mentioned, the other terms can be handled in a similar manner, so we omit those details here.\\
 \indent This shows that $(\overline{u}, \overline{\varphi})$ fulfills the state problem. Thus, 
\begin{equation}
\begin{aligned}
- \infty < \inf_{(u, \varphi) \in \mathcal{F}_{\textup{ad}}} J^{\varepsilon}(u, \varphi) \leqslant J^{\varepsilon}(\overline{u}, \overline{\varphi}) \leqslant \lim_{j \rightarrow \infty}  J^{\varepsilon}(u_j, \varphi_j) =\inf_{(u, \varphi) \in \mathcal{F}_{\textup{ad}}} J^{\varepsilon}(u, \varphi),
\end{aligned}
\end{equation}	
which proves that $(\overline{u}, \overline{\varphi})$ is a minimizer of the phase-field problem.
\end{proof}
\subsection{Reduced cost functional}
Let $\varphi \in \Phi_{\textup{ad}}$ and let $u=S(\varphi) \in U$ be the corresponding state. We introduce the reduced cost functional
\begin{equation}\label{reduced_cost}
\begin{aligned}
j_{\varepsilon}(\varphi):=J^\varepsilon(S(\varphi), \varphi)=J_0(S(\varphi))+\gamma E^\varepsilon(\varphi),
\end{aligned}
\end{equation}
where 
\begin{equation}
J_0(u)=\frac{1}{2}\int_0^T \int_D (u-u_d)^2 \, \textup{d}x \textup{d}s.
\end{equation}
We claim that the reduced objective is Fr\'echet differentiable.
\begin{proposition}
Under the assumptions of Theorem~\ref{Thm:Wellposedness}, the reduced cost functional $j_\varepsilon: \Phi_{\textup{ad}} \rightarrow \R$ is Fr\'echet differentiable.
\end{proposition}
\begin{proof}
The proof follows analogously to the proof of \cite[Lemma 4.2]{blank2014relating}. We first show that $J^\varepsilon: U \times \Phi_{\textup{ad}} \rightarrow \R$ is Fr\'echet differentiable. Formally, the partial derivatives of $J^\varepsilon$ at $(u, \varphi)$ in the direction $(v, h)$ are given by
\begin{equation}
\begin{aligned}
J^\varepsilon_{u}=\, J_{0, u}(u)v,\quad
J^\varepsilon_{\varphi}=\, \gamma E^\varepsilon_{\varphi}(\varphi)h,
\end{aligned}
\end{equation}
with	
\begin{equation}
\begin{aligned}
J_{0, u}(u)v =&\, \int_0^T \int_D (u-u_d)v\, \textup{d}x \textup{d}s, \\
E^\varepsilon_{\varphi}(\varphi)h=&\, \varepsilon \intO \nabla \varphi \cdot \nabla h \dx+ \frac{1}{\varepsilon} \Psi'(\varphi) h \dx.
\end{aligned}
\end{equation}
In order to show the desired Fr\'echet differentiability, it is enough to prove the continuities of the formally computed directional derivatives. To this end, let $\{(u_n, \varphi_n)\} \subset U \times \Phi_{\textup{ad}}$ be a given sequence, such that 
\begin{equation}
(u_n, \varphi_n) \rightarrow (u, \varphi) \text{ in} \ U \times \Phi_{\textup{ad}} \ \text{ as} \ n \rightarrow \infty.
\end{equation}
Then there is a subsequence, which we do not relabel, such that 
\begin{equation}
(u_n, \varphi_n) \rightarrow (u, \varphi) \text{ a.e.\ in } \ \Omega \times (0,T).
\end{equation}
Then
\begin{equation}
\begin{aligned}
\left|J_{0, u}v-J_{0, u_n}v\right|=&\, \left| \int_0^T \int_D (u-u_n)v\, \textup{d}x \textup{d}s \right|\\
 \leqslant&\, \nLtwoLtwo{u-u_n}\nLtwoLtwo{v} \rightarrow 0 \text{ as } n \rightarrow \infty,
\end{aligned}
\end{equation}
and continuity of $E^\varepsilon_{\varphi}(\varphi)h$ follows similarly. Therefore by chain rule
\begin{equation}
\begin{aligned}
j_\varepsilon'(\varphi)h=J_{0, u}(u, \varphi) u^*+J^\varepsilon_{\varphi}(u, \varphi)h,
\end{aligned}
\end{equation}
where $u^*=S'(\varphi) h$; cf. Theorem~\ref{Thm:ParameterToState}.
\end{proof}
\section{Analysis of the adjoint problem} \label{Sec:Adjoint}
We next proceed to studying the adjoint problem, which is is formally given by
\begin{equation} \label{adjoint}
\begin{aligned}
\begin{multlined}[t]
\sca{\alpha(t) p_{tt}}{v}_{\Omega} +\intO (c^2(\varphi) \nabla p(t)-b(\varphi) \nabla p_t) \cdot \nabla v \dx \\
= -\intO ((4k(\varphi)u_{tt}+\alpha_{tt})p+2(\alpha_{t}+2k(\varphi)u_t)p_t) v\dx + \int_D (u-u_d)v \dxs ,
\end{multlined}
\end{aligned}
\end{equation}
for all $v \in H^1(\Omega)$ a.e.\ in time, with $(p, p_t)\vert_{t=T}=(0,0)$. Under an additional regularity assumption on the coefficient $\alpha$, this problem is well-posed.
\begin{proposition} \label{Prop:AdjointWellP}
	Let assumptions of Theorem~\ref{Thm:Wellposedness} hold. Further, let \[\alpha_{tt} \in L^2(0,T; L^2(\Omega)).\] Then there exists a unique $p \in H^1(0,T; H^1(\Omega))$, which solves \eqref{adjoint}.
\end{proposition}
\begin{proof}
Similarly to before, in the following we will show only a priori estimates towards the proof of existence of weak solutions for the formally derived adjoint problem. The uniform estimates we present here can be combined with classical approximation arguments (Galerkin ansatz) to provide a more rigorous proof of the proposition.\\
\indent	Before obtaining the uniform estimates we will re-write the system \eqref{adjoint} as an initial value problem. In that direction we apply the transformation $t\mapsto T-t,$ use the same notations $\alpha,$ $p,$ $v,$ $c(\varphi),$ $b(\varphi),$ $k(\varphi),$ $u$, and $u_{d}$ to denote the transformed quantities to furnish
	 \begin{equation} \label{adjoint*}
	 \begin{aligned}
	 \begin{multlined}[t]
	 \sca{\alpha(t) p_{tt}}{v}_{\Omega} +\intO (c^2(\varphi) \nabla p(t)+b(\varphi) \nabla p_t) \cdot \nabla v \dx \\
	 = -\intO ((4k(\varphi)u_{tt}+\alpha_{tt})p+2(-\alpha_{t}-2k(\varphi)u_t)\cdot(-p_t)) v\dx \\+ \int_D (u-u_d)v \dx ,
	 \end{multlined}
	 \end{aligned}
	 \end{equation}
	 for all $v \in H^1(\Omega)$, with $(p, p_t)\vert_{t=0}=(0,0)$. We choose $v=p_{t}$ in \eqref{adjoint*} and add $\displaystyle\frac{1}{2}\ddt\int_{\Omega}|p|^{2}\dx=\int_{\Omega}pp_{t} \dx$ to both sides to obtain
	 \begin{equation}\label{testingadj}
	 \begin{aligned}
	 &\begin{multlined}[t]\displaystyle\frac{1}{2}\ddt\int_{\Omega}\alpha|p_{t}|^{2}\dx+\frac{1}{2}\ddt\int_{\Omega}c^{2}(\varphi)|\nabla p|^{2}\dx+\frac{1}{2}\ddt\int_{\Omega}|p|^{2}\dx\\+\int_{\Omega}b(\varphi)|\nabla p_{t}|^{2}\dx \end{multlined}\\
	 &\displaystyle=\begin{multlined}[t]-\int_{\Omega}4k(\varphi)u_{tt}pp_{t}\dx-\int_{\Omega}\alpha_{tt}pp_{t}\dx-\int_{\Omega}2\alpha_{t}p_{t}^2\dx-\int_{\Omega}4k(\varphi)u_{t}|p_{t}|^{2}\dx\\+\int_D (u-u_d)p_t\dx
	 \displaystyle +\frac{1}{2}\int_{\Omega}\alpha_{t}|p_{t}|^{2}\dxs+\int_{\Omega}pp_{t} \dx=\sum_{i=1}^{7} J_{i}.\end{multlined}
	 \end{aligned}
	 \end{equation}
	 Integrating \eqref{testingadj} in $(0,\tau)$ for a.e.\ $\tau\in(0,T)$ we obtain
	 \begin{equation}\label{testingadjint}
	 \begin{aligned}
	 &\begin{multlined}[t]\frac{1}{2}\int_{\Omega}\alpha|p_{t}|^{2}(\tau)\dx+\frac{1}{2}\int_{\Omega}c^{2}(\varphi)|\nabla p|^{2}(\tau)\dx+\frac{1}{2}\int_{\Omega}|p|^{2}(\tau)dx\\
	 +\int_{0}^{\tau}\int_{\Omega}b(\varphi)|\nabla p_{t}|^{2}\dxs \end{multlined}\\
	 &=\begin{multlined}[t]-\int_{0}^{\tau}\int_{\Omega}4k(\varphi)u_{tt}pp_{t}\dxs-\int_{0}^{\tau}\int_{\Omega}\alpha_{tt}pp_{t}\dxs+\int_{0}^{\tau}\int_{\Omega}2\alpha_{t}p_{t}^2\dxs\\-\int_{0}^{\tau}\int_{\Omega}4k(\varphi)u_{t}|p_{t}|^{2}\dxs
	 \displaystyle+\int_{0}^{\tau}\int_D (u-u_d)p_t\dx
	  +\frac{1}{2}\int_{0}^{\tau}\int_{\Omega}\alpha_{t}p_{t}^{2}\dxs	\\  +\int_{0}^{\tau}\int_{\Omega}pp_{t}\dxs=\sum_{i=1}^{7} \int_{0}^{\tau}J_{i}\ds.\end{multlined}
	 \end{aligned}
	 \end{equation}
	 In the following we estimate $\displaystyle\int_{0}^{\tau} J_{i} \ds$. The generic constants $C$ in the following inequalities might differ from one line to another and they can depend on  $\|(u_{0},u_{1})\|_{H^{1}(\Omega)\times H^{1}(\Omega)},$ $\|g\|_{H^1(H^{-1/2})},$ $\underline{\alpha},$ $\overline{\alpha},$ $c_{i},$ $k_{i}$, $b_{i}$ for $i\in\{f,l\}$ and the final time $T.$\\
	 \indent We start by estimating $\displaystyle\int_{0}^{\tau} J_{1} \ds$. 	 
	 We will rely on Ehrling's lemma for the following estimates. 
	 \begin{equation}\label{J1}
	 \begin{aligned}
	 \displaystyle\bigg|\int_{0}^{\tau} J_{1}  \ds\bigg|
	 &\displaystyle\leqslant C\int_{0}^{\tau}\|u_{tt}\|_{L^{2}}\|p\|_{L^{4}}\|p_{t}\|_{L^{4}}\ds\\
	 &\displaystyle \leqslant C\int_{0}^{\tau}\|u_{tt}\|^{2}_{L^{2}}\|p\|^{2}_{L^{4}}\ds+C\int_{0}^{\tau}\|p_{t}\|^{2}_{L^{4}}\ds\\
	 &\displaystyle \leqslant C\|p\|^{2}_{L^{\infty}(L^{4})}\|u_{tt}\|^{2}_{L^{2}(L^{2})}+C\|p_{t}\|^{2}_{L^{2}(0,\tau;L^{4})}\\
	 & \displaystyle \leqslant C\|u_{tt}\|^{2}_{L^{2}(L^{2})}\bigg(\|p\|^{2}_{L^{\infty}(L^{2})}+\|{c(\varphi)}\nabla p\|^{2}_{L^{\infty}(L^{2})}\bigg)\\
	 &\displaystyle \qquad +C\left(C(\epsilon)\|\sqrt{\alpha}p_t\|^2_{L^2(L^{2})}+\epsilon\|\sqrt{b(\varphi)}\nabla p_t\|^2_{L^2(L^{2})}\right).
	 \end{aligned}
	 \end{equation}
	 Similarly, $\displaystyle\int_{0}^{\tau}J_{2}\, \textup{d}\tau$ can be estimated to have
	\begin{equation}\label{J2}
	\begin{array}{ll}
	&\displaystyle\bigg|\int_{0}^{\tau} J_{2} \ds\bigg|\\
	& \displaystyle \leqslant C\|\alpha_{tt}\|^{2}_{L^{2}(L^{2})}\bigg(\|p\|^{2}_{L^{\infty}(L^{2})}+\|{c(\varphi)}\nabla p\|^{2}_{L^{\infty}(L^{2})}\bigg)\\
	&\displaystyle \qquad +C\left(C(\epsilon)\|\sqrt{\alpha}p_t\|^2_{L^2(L^{2})}+\epsilon\|\sqrt{b(\varphi)}\nabla p_t\|^2_{L^2(L^{2})}\right).
	\end{array}
	\end{equation} 
for some positive parameter $\epsilon.$ The term $\displaystyle\int\limits_{0}^{\tau}J_{3}$ is estimated as follows:
\begin{equation}\label{J3}
\begin{array}{ll}
\displaystyle\bigg|\int_{0}^{\tau} J_{3} \ds\bigg|
&\displaystyle \leqslant \|\alpha_{t}\|_{L^{\infty}(L^{2})}\|p_{t}\|^{2}_{L^{2}(L^{4})}\\
&\displaystyle \leqslant C\bigg(\|\alpha_{t}\|_{L^{\infty}(L^{2})}\|\sqrt{\alpha}p_{t}\|^{2}_{L^{2}(L^{2})} +\epsilon\|\alpha_{t}\|_{L^{\infty}(L^{2})}\|\sqrt{b(\varphi)}\nabla p_{t}\|^{2}_{L^{2}(L^{2})}\bigg).
\end{array}
\end{equation}
Since $k(\varphi)\in L^{\infty}(\Omega),$ the term $\displaystyle \int\limits_{0}^{\tau}J_{4}$ can exactly be estimated as $\displaystyle \int\limits_{0}^{\tau}J_{3}$, only by replacing the role of $\alpha_{t}$ by $u_{t}.$
The term $\displaystyle \int\limits_{0}^{\tau}J_{5}$ can be easily estimated as follows:
\begin{equation}
\begin{array}{l}
\displaystyle\bigg|\int_{0}^{\tau}J_{5} \ds\bigg|\leqslant C\bigg(\|(u-u_{d})\|^{2}_{L^{2}(L^{2})}+\|\sqrt{\alpha}p_{t}\|^{2}_{L^{2}(L^{2})}\bigg)
\end{array}
\end{equation}
The term $\displaystyle\int\limits_{0}^{\tau}J_{6}$ admits similar estimate as that of  $\displaystyle\int\limits_{0}^{\tau}J_{3}.$\\
Next $\displaystyle\int_{0}^{\tau} J_{7}$ is estimated as
$$\displaystyle\bigg|\int_{0}^{\tau}J_{7} \ds\bigg|\leqslant C\bigg(\|p\|^{2}_{L^{2}(L^{2})}+\|\sqrt{\alpha}p_{t}\|^{2}_{L^{2}(L^{2})}\bigg)$$
Combining all the above estimates for $\displaystyle\int\limits_{0}^{\tau}J_{i}$ ($i=1,..,7$) together with \eqref{testingadjint}, and using Gr\"{o}nwall's inequality completes the proof of Proposition~\ref{Prop:AdjointWellP}.
\end{proof}
\section{First-order necessary optimality conditions} \label{Sec:FirstOrderOpt}
We next wish to derive the first-order optimality conditions. To this end, let $\varphi \in \Phi_{\textup{ad}}$ be the minimizer of the optimal control problem \eqref{objective}--\eqref{state_constraint_strong} and $u=S(\varphi)$ the associated state variable. With the reduced cost function $j_\varepsilon$ given in \eqref{reduced_cost}, the optimization problem can be reformulated as follows:
\begin{equation} \label{reduced_problem}
\min_{\varphi \in \Phi_{\textup{ad}}} j_\varepsilon(\varphi).
\end{equation}
\begin{proposition} Let the assumptions of Theorem~\ref{Thm:Wellposedness} and Proposition~\ref{Prop:AdjointWellP} hold. Let $u^* \in X$ be the solution to \eqref{diff_state} and let $p \in H^1(0,T; H^1(\Omega))$ be the adjoint state defined as the weak solution to \eqref{adjoint}. Then
\begin{equation}\label{J_u_diff}
\begin{aligned}
	J^\varepsilon_{u}(u, \varphi) u^* =&\, \begin{multlined}[t]-
\int_0^T\intO (2c(\varphi)c'(\varphi)h \nabla u(t)+b'(\varphi)h \nabla u_t(t)) \cdot \nabla p \dxs\\+\int_0^T\intO 2k'(\varphi)h u_t^2(t) p\dxs. \end{multlined}
\end{aligned}
\end{equation}
\end{proposition}
\begin{proof}
	Testing the adjoint problem with $u^* \in X$ gives
	\begin{equation} 
	\begin{aligned}
	\begin{multlined}[t]
	\int_0^T\sca{\alpha(t) p_{tt}}{u^*}_{\Omega}\ds +\int_0^T\intO (c^2(\varphi) \nabla p(t)-b(\varphi) \nabla p_t) \cdot \nabla u^* \dxs\\
	+\int_0^T\intO \big\{(4k(\varphi)u_{tt}+\alpha_{tt})p+2(\alpha_{t}+2k(\varphi)u_t)p_t\big\} u^*\dxs=		J^\varepsilon_{u}(u, \varphi) u^*.
	\end{multlined}
	\end{aligned}
	\end{equation}
	Integration by parts with respect to time then leads to
		\begin{equation} 
	\begin{aligned}
	\begin{multlined}[t]
	\int_0^T\sca{\alpha(t) u^*_{tt}}{p}_{\Omega}\ds +\int_0^T\intO (c^2(\varphi) \nabla u^*(t)+b(\varphi) \nabla u^*_t(t)) \cdot \nabla p \dxs\\
	+\int_0^T\intO 4k(\varphi)\left\{u_{tt}p+u_tp_t\right\} u^*\dxs=		J^\varepsilon_{u}(u, \varphi) u^*.
	\end{multlined}
	\end{aligned}
	\end{equation}
Using the fact that $u^*$ satisfies the weak form \eqref{diff_state} then yields
\begin{equation}
\begin{aligned}
&\begin{multlined}[t]
-\int_0^T\intO 4k(\varphi)\left\{u_{tt}p+u_tp_t\right\} u^*\dxs+		J^\varepsilon_{u}(u, \varphi) u^* \end{multlined}\\
=&\, \begin{multlined}[t]-\int_0^T\intO (2c(\varphi)c'(\varphi)h \nabla u(t)+b'(\varphi)h \nabla u_t(t)) \cdot \nabla p \dxs \\+\int_0^T\intO 4k(\varphi) u_t u_t^*(t) p\dxs+\int_0^T\intO 2k'(\varphi)h u_t^2(t) p\dxs,  \end{multlined}
\end{aligned}
\end{equation}	
from which \eqref{J_u_diff} follows by integration by parts with respect to time in the first term on the left.
\end{proof}	
\begin{proposition}
Let $\varphi \in \Phi_{\textup{ad}}$ be the solution to \eqref{reduced_problem}. Then
\[
j_\varepsilon'(\varphi) (\tilde{\varphi}-\varphi) \geq 0 \qquad \forall \tilde{\varphi} \in \Phi_{\textup{ad}},
\]
where
\begin{equation}
\begin{aligned}
j_\varepsilon'(\varphi)(\tilde{\varphi}-\varphi)
=&\,\begin{multlined}[t] J^{\varepsilon}_{\varphi}(u, \varphi)(\tilde{\varphi}-\varphi)\\-\int_0^T\intO (2c(\varphi)c'(\varphi)(\tilde{\varphi}-\varphi)  \nabla u(t)+b'(\varphi)(\tilde{\varphi}-\varphi)  \nabla u_t(t)) \cdot \nabla p \dxs\\+\int_0^T\intO 2k'(\varphi)(\tilde{\varphi}-\varphi)  u_t^2(t) p\dxs
\end{multlined}
\end{aligned}
\end{equation}
is the directional derivative of the reduced objective functional $j_\varepsilon$.
\end{proposition}

\noindent By collecting previous results, we arrive at the optimality system. 
\begin{theorem}[Optimality system]
Let $\varphi \in \Phi_{\textup{ad}}$ be the minimizer of the optimal control problem \eqref{objective}--\eqref{state_constraint_strong} and $u=S(\varphi)$ and $p$ the associated state and adjoint variables, respectively. Then the functions $(u, \varphi, p) \in U \times \Phi_{\textup{ad}} \times H^1(0,T; H^1(\Omega))$ satisfy the following optimality system in the weak sense: \\
\noindent {The state problem}
\begin{equation} 
\begin{aligned} 
 \begin{cases} 
\alpha(x,t) u_{tt}-\textup{div}(c^2(\varphi) \nabla u)- \textup{div}(b(\varphi) \nabla u_t)= 2k(\varphi)u_t^2\ \ \text{in} \ \ \Omega \times (0,T), \\[2mm]
c^2(\varphi)\dfrac{\partial u}{\partial n}+b(\varphi) \dfrac{\partial u_t}{\partial n}=g\ \ \text{on} \ \ \Gamma,\\[2mm]
(u, u_t)\vert_{t=0}=(u_0, u_1).
\end{cases}
\end{aligned}
\end{equation}
\noindent {The adjoint problem}
\begin{equation} 
\begin{aligned} 
\begin{cases} 
\alpha(x,t) p_{tt}-\textup{div}(c^2(\varphi) \nabla p)+ \textup{div}(b(\varphi) \nabla p_t)\\[2mm]= -(4k(\varphi)u_{tt}+\alpha_{tt})p-2(\alpha_{t}+2k(\varphi)u_t)p_t)+(u-u_{\textup{d}})\chi_D\ \ \text{in} \ \ \Omega \times (0,T), \\[2mm]
c^2(\varphi)\dfrac{\partial p}{\partial n}-b(\varphi) \dfrac{\partial p_t}{\partial n}=0 \ \ \text{on} \ \ \Gamma,\\[2mm]
(p, p_t)\vert_{t=T}=(0, 0).
\end{cases}
\end{aligned}
\end{equation}
\noindent {The gradient inequality}
\begin{equation}
\begin{aligned}
\,\begin{multlined}[t] \gamma \varepsilon \int_{\Omega} \nabla \varphi \cdot \nabla (\tilde{\varphi}-\varphi) \dx +\dfrac{\gamma}{\varepsilon}\int_{\Omega} \Psi'(\varphi)(\tilde{\varphi}-\varphi) \dx\\-\int_0^T\intO (2c(\varphi)c'(\varphi)(\tilde{\varphi}-\varphi)  \nabla u(t)+b'(\varphi)(\tilde{\varphi}-\varphi)  \nabla u_t(t)) \cdot \nabla p \dxs\\+\int_0^T\intO 2k'(\varphi)(\tilde{\varphi}-\varphi)  u_t^2(t) p\dxs \geqslant 0 \qquad \forall \tilde{\varphi} \in \Phi_{\textup{ad}}.
\end{multlined}
\end{aligned}
\end{equation}
\end{theorem}
\section{Sharp interface limit} \label{Sec:SharpInterface}
In this last section, we discuss the sharp interface limit of $\{j_{\varepsilon}\}_{\varepsilon>0}$. We need the following preparatory result.
\begin{proposition} \label{Prop:ContL1}
Under the assumptions of Theorem \ref{Thm:Wellposedness} and Proposition~\ref{Prop:AdjointWellP}, the mapping
\[
 \varphi \mapsto \frac12 \int_0^T \int_{D} (S(\varphi)-u_d)^2 \dxt
\]
is continuous in $L^1(\Omega)$.	
\end{proposition} 
\begin{proof}
 Let $\{\varphi_n\}_{n \in \mathbb{N}}$ be a sequence in $L^1(\Omega)$, such that
 \[\lim_{n \rightarrow \infty} \|\varphi-\varphi_n\|_{L^1(\Omega)}=0.\]
We can extract a subsequence, not relabeled, which converges a.e.\ to $\varphi$ in $\Omega \times (0,T)$. We define $u_n:=S(\varphi_n)$ and note that it satisfies 
	\begin{equation}\label{weakformn}
\begin{aligned}
\begin{multlined}[t]
\intO	\alpha(t) u_{n, tt}(t) v \dx+\intO (c^2(\varphi_n) \nabla u_n(t)+b(\varphi_n) \nabla u_{n, t}(t)) \cdot \nabla v \dx \\
= \intO 2k(\varphi_n) u_{n,t}^2(t) v\dx+ \int_{\Gamma} g(t) v \dx \quad \text{a.e.\ in } (0,T),
\end{multlined}
\end{aligned}
\end{equation}
for all $v \in H^1(\Omega)$, with $(u_n, u_{n, t})\vert_{t=0}=(u_0, u_1)$. By Theorem~\ref{Thm:Wellposedness}, we know that there exists $C>0$, independent of $n$, such that
\[
\|u_{n}\|_{U} \leqslant C.
\]
Thus we can find a subsequence, which we do not relabel, such that $\{u_n\}_{n \in \N}$ converges weakly in $H^1(0,T; H^1(\Omega)) \cap H^2(0,T; L^2(\Omega)).$\\
\indent We next show that the sub-sequential limit of $\{u_n\}_{n \in \N}$ can be identified with $S(\varphi)$. In that direction, we investigate the convergence of each summand of \eqref{weakformn}. We will only detail the arguments for the second summand of \eqref{weakformn} and for the rest we will briefly comment since they are relatively straightforward to deal with.\\
\indent Given $v \in H^1(\Omega)$, we have the uniform estimate
 \begin{equation}
 \begin{array}{l}

\displaystyle|c^2(\varphi_n(x)) \nabla v(x)| \leqslant C c^2_f|\nabla v(x)| \quad \text{a.e.\ in }\ \Omega \times (0,T)

\end{array}
\end{equation}
and thus by Lebesgue's dominated convergence theorem $\{c^2(\varphi_n)\nabla v\}_{n \in \N}$ converges strongly in $L^2(\Omega)^d$ to $c^2(\varphi) \nabla v$. Thus, since $\{\nabla u_n\}_{n \in \N}$ converges weakly in $L^2(\Omega)^d$, we obtain
\begin{equation}\label{secondtermconv}
\begin{array}{l}
\displaystyle\lim_{n \rightarrow \infty}\left|\intO c^2(\varphi_n) \nabla u_n(t)\cdot \nabla v \dx-\intO c^2(\varphi) \nabla u(t)\cdot \nabla v \dx \right|=0.
\end{array}
\end{equation}
Further since for a.e.\ $t\in[0,T],$ $\{u_{n,tt}\}_{n \in \N}$ and $\{\nabla u_{n,t}\}_{n \in \N}$ weakly converge to $u_{tt}$ and $\nabla u$, respectively, in 
$L^{2}(\Omega)),$ similar arguments (to the ones leading to \eqref{secondtermconv}) can be applied to show that
\begin{equation}\label{firsttermconv}
\begin{array}{l}
\displaystyle\lim_{n \rightarrow \infty}\left|\intO \alpha(t) u_{n, tt}(t) v \dx-\intO \alpha(t)u_{tt} v \dx \right|=0
\end{array}
\end{equation}
and 
\begin{equation}\label{thirdtermconv}
\begin{array}{l}
\displaystyle\lim_{n \rightarrow \infty}\left|\intO b(\varphi_n) \nabla u_{n, t}(t) \cdot \nabla v \dx-\intO b(\varphi) \nabla u_{t}(t) \cdot \nabla v \right|=0.
\end{array}
\end{equation}
Now let us consider the first term on the right in \eqref{weakformn}. We know that $\{u_{n,t}\}_{n \in \N}$ is uniformly bounded in $L^{2}(0,T;H^{1}(\Omega))\cap H^{1}(0,T;L^{2}(\Omega))$ on account of Theorem~\ref{Thm:Wellposedness}. Hence by Aubin--Lions lemma one can extract a subsequence (not explicitly relabeled) such that $\{u_{n,t}\}_{n \in \N}$ converges strongly to $u_{t}$ in $L^{2}(0,T;L^{4}(\Omega)).$ In turn, this implies that for a.e.\ $t\in[0,T]$, $\{u^{2}_{n,t}\}_{n \in \N}$ converges to $u^{2}_{t}$ strongly in $L^{2}(\Omega).$ Now one can use similar arguments as before to show that 
\begin{equation}\label{fourthtermconv}
\begin{array}{l}
\displaystyle\lim_{n \rightarrow \infty}\left|\intO 2k(\varphi_n) u_{n,t}^2(t) v\dx \dx-\intO 2k(\varphi) u_{t}^2(t) v \dx \right|=0.
\end{array}
\end{equation}
In view of the above convergences we conclude that $(\varphi,u)$ solves the weak formulation \eqref{weakformmainprob} for all $v \in H^1(\Omega)$, with $(u_n, u_{n, t})\vert_{t=0}=(u_0, u_1).$ From the uniqueness of weak solution to the problem \eqref{weakformmainprob}, we conclude that $u=S(\varphi).$\\ 
\indent Next by the compact embedding of the space $W=\{u \in L^2(0,T; H^1(\Omega)):\, u_t \in L^2(0,T; L^2(\Omega)) \}$ into $L^2(0,T; L^2(\Omega))$, weak convergence of $\{u_n\}_{n \in \N}$ to $S(\varphi)$ implies strong convergence in $L^2(0,T; L^2(\Omega))$.\\
\indent Finally, the proof of the proposition follows by the continuity of the cost functional.
\end{proof}
We next discuss the $\Gamma$-convergence of the reduced cost functionals.
\subsection{Sharp interface problem} The sharp interface problem will be given by
\begin{equation} \label{objective_sharp}
\begin{aligned}
\min_{(u, \varphi)} J^0(u, \varphi)= \dfrac12 \int_0^T \int_D (u-u_\textup{d})^2 \, \textup{d}x \textup{d}s+\gamma C_0 P_{\Omega}({\varphi=1}),
\end{aligned}
\end{equation}
where $C_0>0$ is a constant that depends on the choice of the potential, with 
\begin{equation} \label{admissible_sets_sharp}
\varphi \in \Phi^0_{\textup{ad}}=  BV(\Omega, \{0,1\}),\quad u \in U,
\end{equation}
and such that
\begin{equation} \label{state_constraint_strong_sharp}
	\begin{aligned}
\begin{multlined}[t]
\intO	\alpha(t) u_{tt}(t) v \dx+\intO (c^2(\varphi) \nabla u(t)+b(\varphi) \nabla u_t(t)) \cdot \nabla v \dx \\
= \intO 2k(\varphi) u_t^2(t) v\dx+ \int_{\Gamma} g(t) v \dx \quad \text{a.e.\ in } (0,T),
\end{multlined}
\end{aligned}
\end{equation}
for all $v \in H^1(\Omega)$ a.e.\ in time, with $(u, u_t)\vert_{t=0}=(u_0, u_1)$. 
Well-posedness of the state problem \eqref{state_constraint_strong_sharp} for $\varphi \in \Phi_{\textup{ad}}^0$, which corresponds to the sharp interface setting, follows directly from Theorem~\ref{Thm:Wellposedness}. We can thus define the reduced objective functional
\begin{equation}
\begin{aligned}
j_0(\varphi)= \begin{cases}
J^0(S(\varphi), \varphi) \quad & \text{if} \ \varphi \in \Phi_{\textup{ad}}^0, \\
+\infty & \text{otherwise},
\end{cases}
\end{aligned}
\end{equation}
and reformulate the sharp interface problem as
\begin{equation}
\begin{aligned}
\min_{\varphi \in \Phi_{\textup{ad}}^0} j_0(\varphi).
\end{aligned}
\end{equation} 
\begin{theorem}\label{Gammaconv}
Under the assumptions of Theorem \ref{Thm:Wellposedness} and Proposition~\ref{Prop:AdjointWellP}, the functionals $\{j_{\varepsilon}\}_{\varepsilon>0},$ $\Gamma$-converge in $L^{1}(\Omega)$ to $j_{0}$ as $\varepsilon\searrow 0.$ 
\end{theorem}
\begin{proof}
The Ginzburg--Landau energy 
	\begin{equation} 
	\begin{aligned}
	E_\varepsilon(\varphi)=\begin{cases} \displaystyle \int_{\Omega} \dfrac{\varepsilon}{2}|\nabla \varphi|^2+\dfrac{1}{\varepsilon}\Psi(\varphi) \textup{d}x,\quad \varepsilon>0 \quad  \ &\text{if } \varphi \in H^1(\Omega), \\[2mm]
	+ \infty, \hphantom{} &\text{otherwise}, \end{cases} 
	\end{aligned}
	\end{equation}
is known to $\Gamma$-converge as $\varepsilon \rightarrow 0^+$ in $L^1(\Omega)$ to	a multiple of the perimeter functional
\begin{equation}
\begin{aligned}
E_0: L^1(\Omega) \ni \varphi \mapsto \begin{cases}
C_0 P_{\Omega}({\varphi=1}) \quad &\text{if} \ \varphi \in BV(\Omega, \{0, 1\}),\\
+ \infty, \ &\text{else};
 \end{cases}
\end{aligned}
\end{equation}
cf.~\cite{blowey1991cahn, modica1987gradient}. The constant $C_{0}$ has a different expression compared to the one appearing in ~\cite{blowey1991cahn, modica1987gradient}. This is because in our case we have a different expression for the function $\Psi(\cdot)$ and in the present scenario $C_{0}$ is given by the following expression:
\begin{equation}\label{C0}
\begin{array}{l}
\displaystyle C_{0}=\int_{0}^{1}\sqrt{2\Psi(s)}\ds=\frac{\pi}{8}.
\end{array}
\end{equation} Following~\cite[Theorem 20]{blank2016sharp}, we then write the reduced objective as a sum,
\begin{equation}
\begin{aligned}
j_{\varepsilon}= \gamma E_\varepsilon+\frac12 \int_0^T \int_{D} (S(\varphi)-u_d)^2 \dxt+I_{\Phi_{\textup{ad}}}
\end{aligned}
\end{equation}
with $I_{\Phi_{\textup{ad}}}(\varphi)=0$ if $\varphi \in \Phi_{\textup{ad}}$ and $I_{\Phi_{\textup{ad}}}(\varphi)=+ \infty$ if $\varphi \in L^1(\Omega) \setminus \Phi_{\textup{ad}}$. Using the results of~\cite{modica1987gradient}, we also find that $\gamma E_\varepsilon+I_{\Phi_{\textup{ad}}}$ $\Gamma$-converges to $\gamma E_0+I_{\Phi_{\textup{ad}}}$. \\
\indent On account of Proposition~\ref{Prop:ContL1}, $j_{\varepsilon}$ is a sum of $\gamma E_\varepsilon+I_{\Phi_{\textup{ad}}}$ and a functional that is continuous in $L^1(\Omega)$. Thus by,  e.g.~\cite[Proposition 20]{dal2012introduction}, it follows that $(j_\varepsilon)_{\varepsilon>0}$ $\Gamma$-converges in $L^1(\Omega)$ to $j_0$.
\end{proof}
Relying on the above results, one can use the compactness arguments of~\cite{modica1987gradient} and methods from the theory of $\Gamma$-convergence, see~\cite{blowey1991cahn}, to show the following result.
\begin{corollary}
	Suppose the assumptions of Theorem \ref{Thm:Wellposedness} and Proposition~\ref{Prop:AdjointWellP} hold true. Further, let $\{\varphi_{\varepsilon}\}_{\varepsilon>0}$ be minimizers of $\{j_{\varepsilon}\}_{\varepsilon>0}.$ Then there exists a subsequence (not explicitly relabeled) and an element $\varphi_{0}\in L^{1}(\Omega),$ such that \[\lim\limits_{\varepsilon\searrow 0}\|\varphi_{\varepsilon}-\varphi_{0}\|_{L^{1}(\Omega)}=0.\] Besides, $\varphi_{0}$ is a minimizer of $j_{0}$ and $\lim\limits_{\varepsilon\searrow 0} j_{\varepsilon}(\varphi_{\varepsilon})=j_{0}(\varphi_{0}).$ 
	
\end{corollary}

\section*{Acknowledgments}
S.M. received funding from the Alexander von Humboldt foundation. The support is gratefully acknowledged.

\bibliographystyle{abbrv} 
\bibliography{references}
\end{document}